\documentclass{amsart}
\usepackage[english]{babel}
\usepackage[all]{xy}
\usepackage{amssymb,latexsym}
\usepackage{mltex}
\usepackage{amsmath}
\usepackage[colorlinks=true,linkcolor=blue,citecolor=blue]{hyperref}
\usepackage{lscape}
\usepackage{enumerate}
\usepackage{amsthm}
\usepackage{hyperref}
\usepackage{multicol}
\newtheorem{thm}{Theorem}
\newtheorem{cor}{Corollary}
\newtheorem{lem}{Lemma}[section]
\newtheorem{prop}[lem]{Proposition}
\newtheorem{rem}{Remark}

\newcommand{\C}{\mathbb{C}}

\newcommand{\R}{\mathbb{R}}

\newcommand{\beqt}{\begin{equation}}  \newcommand{\eeqt}{\end{equation}}
\newcommand{\bal}{\begin{align}}      \newcommand{\eal}{\end{align}}
\newcommand{\ba}{\begin{array}}      \newcommand{\ea}{\end{array}}
\newcommand{\bc}{\begin{center}}     \newcommand{\ec}{\end{center}}
\newcommand{\be}{\begin{enumerate}}  \newcommand{\ee}{\end{enumerate}}
\newcommand{\beq}{\begin{eqnarray}}  \newcommand{\eeq}{\end{eqnarray}}
\newcommand{\beQ}{\begin{eqnarray*}} \newcommand{\eeQ}{\end{eqnarray*}}
\newcommand{\bi}{\begin{itemize}}    \newcommand{\ei}{\end{itemize}}
\newcommand{\bt}{\begin{tabular}}    \newcommand{\et}{\end{tabular}}

\newcommand{\g}{\mathfrak{g}}

\begin{document}
	\title{Spinorial representation of surfaces in Lorentzian homogeneous spaces of dimension $3$}
	\author{Berenice Zavala Jim\'enez}
	\begin{abstract}
		We find a spinorial representation of a Riemannian or Lorentzian surface in a Lorentzian homogeneous space of dimension $3.$ We in particular obtain a representation theorem for surfaces in the $\mathbb{L}(\kappa,\tau)$ spaces. We then recover the Calabi correspondence between minimal surfaces in $\mathbb{R}^3$ and maximal surfaces in $\mathbb{R}_1^3$, and obtain a new Lawson type correspondence between CMC surfaces in $\mathbb{R}_1^3$ and in the 3-dimensional pseudo-hyperbolic space $\mathbb{H}_1^{3}.$
	\end{abstract}
	\maketitle
	
	\noindent
	
	\date{}
	%%%%%%%%%%%%%%%%%%%%%%%%%%%%%%%%%%%%%%%%%%%%%%%%%%%%%%%%%%%%%%%%%%%%%%%%%%%%%%%%%%%
	\maketitle\pagenumbering{arabic}
	\section{Introduction}
	A 3-dimensional Lorentzian homogeneous space, simply connected and complete, is a Lie group with a left-invariant Lorentzian metric, or one of the following symmetric spaces (see \cite{GC}, \cite{CP}, \cite{CW}, \cite{WOLF},\cite{RW1} and \cite{RW2} for a detailed description of this classification): 
	\begin{itemize}
		\item the Minkowski space $\R_1^3,$ the pseudo-sphere ${\mathbb{S}_1^3}$ (the de Sitter space), the universal covering of the pseudo-hyperbolic space $\widetilde{\mathbb{H}_1^3}$ (the anti-de Sitter space);
		\item the products $\R_{-}\times \mathbb{S}^2$, $\R_{-}\times \mathbb{H}^2$, $\R\times \widetilde{\mathbb{S}_1^2}$ and $\R\times \widetilde{\mathbb{H}_1^2},$ where $\widetilde{\mathbb{S}_1^2}$  and $\widetilde{\mathbb{H}_1^2}$ are the universal coverings of the pseudo-sphere and the pseudo-hyperbolic space of dimension 2;
		\item the Cahen-Wallach spaces 
		$$M_c:=(\mathbb{R}^3,g_{_{c}}),\ {g_{_{c}}}_{(s,t,x)}:= 2ds dt + c x^2 ds^2+ dx^2,\\\ c=\pm 1.$$
	\end{itemize}
	Note that since the second fundamental group of a Lie group is necessarily trivial (see \cite[Pag. 116-118]{MILMORSE}), in this list the symmetric spaces $\mathbb{S}^3_1$ and $ \R_{-} \times\mathbb{S}^2$ do not admit a Lie group structure.

	The aim of this paper is to characterize the immersion of a surface in a given simply connected complete Lorentzian homogeneous space of dimension $3$ in terms of spinors. We will focus here on the Lorentzian metric Lie groups and on the symmetric spaces of constant curvature or which are a product, and leave the case of the Cahen-Wallach spaces to a further study. We will moreover pay a special attention to the family $\mathbb{L}(\kappa,\tau)$ of homogeneous spaces with a 4-dimensional group of isometries.
	
	The similar problem in the Riemannian setting is entirely solved: Friedrich characterized in \cite{TFI} the immersions in $\R^3$, Morel in \cite{MOR} the immersions in $\mathbb{S}^3$ and $\mathbb{H}^3$ and Roth in \cite{JRH} the immersions in Riemannian homogeneous spaces with a $4$-dimensional isometry group; we then found in \cite{BRZ} a characterization of the immersions in Riemannian metric Lie groups (of arbitrary dimension and co-dimension): since a simply connected Riemannian homogeneous space of dimension $3$ is isometric to a metric Lie group or is the product manifold $\mathbb{S}^2\times\R$ (see \cite{MP,Mi}), we could achieve the characterization in the 3-dimensional Riemmanian homogeneous spaces. 
	
	Besides the natural question of the extension to the Lorentzian setting of results obtained in Riemannian geometry, our motivation is to understand natural geometric correspondences between CMC surfaces in some Riemmanian or Lorentzian 3-dimensional homogeneous manifolds, as for instance the Lawson-type correspondences described in \cite{BD} and their Lorentzian analogues: to that end the spinorial approach seems promising since it permits to obtain a very concise formalism for the surface theories in these spaces. In that direction, we give at the end of the paper a Lawson-type correspondence between CMC surfaces in $\R^3_1$ and in $\mathbb{H}^3_1,$ and we find a simple new proof of the correspondence between minimal surfaces in $\R^{3}$ and maximal surfaces in $\R^3_1$. 
	
	Let us quote some papers concerning immersions in 3-dimensional semi-Rieman\-nian manifolds: Lawn \cite{La,LaT} studied immersions in $\mathbb{R}_1^3,$ Lawn and Roth \cite{LR} in 3-dimensional space forms and Roth \cite{JR} in Lorentzian products, especially in $M^2(k)\times \R_{-}.$ Bayard \cite{Bay}, Bayard and Patty \cite{BP} and Patty \cite{Pat} obtained afterwards similar results in 3-dimensional Lorentzian space forms as a consequence of their study of immersions in dimension 4. We will recover some of these results using explicitly the metric Lie group structures.
	
	The paper is organized as follows: we present the spinorial representation of a Riemannian or a Lorentzian surface in a Lorentzian metric Lie group of dimension 3 in Sections \ref{preliminares_lie} (via spinors of the group) and \ref{version_intrinseca} (via spinors of the surface). We then study the immersions in Lorentzian products which have a group structure in Section \ref{products}, and in $\mathbb{S}_1^3$ and $\mathbb{H}_1^3$ in Section \ref{sitter_antisitter}. We then study the product $\R_-\times\mathbb{S}^2$ in Section \ref{product_nogrupo} and the $\mathbb{L}(\kappa,\tau)$ spaces in Section \ref{inmersiones_lkt}. We finally describe in Section \ref{sec corr} a Lawson-type correspondence between CMC surfaces in $\R_1^3$ and in $\mathbb{H}_1^3,$ and give a spinorial proof of the correspondence between minimal surfaces in $\R^{3}$ and maximal surfaces in $\R_1^3$. A short appendix ends the paper, first on the identification of intrinsic and extrinsic spinor bundles, and then on the representation of skew-symmetric operators using the Clifford algebra.

	\section{Representation of a surface into a metric Lie group }\label{preliminares_lie}
	\subsection{Spinor bundle of a Lie group $G$}
	Let $G$ be a simply connected Lorentzian Lie group of dimension 3 endowed with a left-invariant metric $\langle \cdot ,\cdot \rangle$ and $\g$ its Lie algebra. The $\g$-valued 1-form
	\begin{eqnarray}\label{maurer_cartan}
	\omega_g(v)=dL_{_{g^{-1}}}(v),& v\in T_gG
	\end{eqnarray}
	is the Maurer-Cartan form and we use it to give a trivialization of the tangent bundle of $G$
	\begin{eqnarray*}\label{trivializacion_TG}
	TG&\longrightarrow& G\times \g
	\\
	X&\longmapsto& (g,\omega_g(X)).\nonumber
	\end{eqnarray*}
	
	A section $X$ of $TG$ is a map $G\longrightarrow \g,$ and we will say that it is a left-invariant section if it is a constant function. Let us denote by $\overline{\nabla}$ the Levi-Civita connection of $G$ and by $\Gamma:\g\longrightarrow \Lambda^2\g$ the application which maps a left-invariant section $X\in\g$ to the skew-symmetric linear transformation $\Gamma(X)$ of $\g$ such that
	\begin{equation}\label{gama}
	\overline{\nabla}_{X}Y=\Gamma(X)(Y)
	\end{equation}
	for all left-invariant section $Y \in \Gamma(TG).$ Using that the connection $ \overline{\nabla}$ is torsion free, we have, for all $X,Y\in {\g},$
	\begin{equation*}\label{nablaG_notorsion}
	\Gamma(X)(Y)-\Gamma(Y)(X)=[X,Y].
	\end{equation*}
	
	Since the tangent bundle of $G$ is trivial, its  orthonormal frame bundle is $Q_G=G\times SO(\g)$ and $\widetilde{Q}_G=G\times Spin(\g)$ is a spin structure over that bundle. The spin group has a natural representation in the Clifford algebra $Cl(\g)$ given by the left-multiplication
	\begin{eqnarray*}\label{rep_haz_espinoresG}
	\rho: Spin(\g)&\longrightarrow&Gl(Cl({\g}))\\
	x&\longmapsto&\begin{array}[t]{rcc}
	\rho(x):Cl({\g})&\longrightarrow &Cl({\g})\nonumber\\
	v &\longmapsto& x\cdot v.
	\end{array}
	\end{eqnarray*}
	We define the spinor bundle 
	\begin{eqnarray}\label{haz_espinoresG}
	\Sigma G:= \widetilde{Q}_G\times_\rho Cl(\g),
	\end{eqnarray}
	associated to the spin structure $\widetilde{Q}_G$ (we identify $(p,q)$ and $(p',q')$ belonging to $\widetilde{Q}_G\times Cl(\g)$ if there exists $g\in Spin(\g)$ such that $(p',q')=(p\cdot g,\rho(g^{-1})q)$). Note that this is not the usual spinor bundle since the representation $\rho$ is not the usual spin representation (it is not irreducible and not complex).

	\begin{rem}\label{G_espinores_trivial}Since $\widetilde{Q}_G=G\times Spin(\g)$ is the trivial spin structure there is a natural isomorphism between the spinor bundle defined in \eqref{haz_espinoresG} and the trivial bundle $G\times Cl(\g)$.
	\end{rem}
	Using the adjoint representation
	\begin{equation*}\label{rep_clifford_bundle}
	Spin(\g) \overset{Ad}{\longrightarrow}SO(\g)\longrightarrow Gl(Cl(\g))
	\end{equation*}
	we define as usual the Clifford bundle
	\begin{equation*}\label{clifford_bundle}
	Cl_{\Sigma G}:= \overset{\sim}{Q}\times_{Ad} Cl(\g)
	\end{equation*}
	and the Clifford action $\cdot:Cl_{\Sigma G}\times \Sigma G \longrightarrow  \Sigma G$.
	
	In a slightly different context, the covariant derivative on $\Sigma G$ is defined in \cite[Thm. 2.7]{BHMM} as follows: if $\varphi$ is a section of $\Sigma G$, $X\in TG$ and  $\partial$ is the usual derivative in the trivialization of Remark \ref{G_espinores_trivial}, then the covariant derivative is
	\begin{equation}\label{derivative_spinor}
	\nabla^G_X \varphi= \partial_X\varphi +\frac{1}{2}\sum_{1\leq k < l \leq 3}\varepsilon_k \varepsilon_l \langle  \overline{\nabla}_X e_k,e_l \rangle e_k\cdot e_l \cdot \varphi,
	\end{equation}
	where $(e_1,e_2, e_3)$ is an orthonormal frame of $G$, $\varepsilon_j=\langle e_j,e_j\rangle =\pm 1$ and $\overline{\nabla}$ is the Levi-Civita connection in $G$. 
	
	Let us call a section $\varphi$ of $\Sigma G$ a left-invariant spinor field if it is constant when regarded as a function from $G$ to $Cl(\g)$. By \eqref{derivative_spinor} the covariant derivative of a left-invariant spinor field $\varphi$ is
	\begin{equation}\label{derivada_grupo}
	\nabla^G_X\varphi=\frac{1}{2}\Gamma(X)\cdot \varphi,
	\end{equation}
	where $\Gamma(X) \in \Lambda^2\g \subset Cl(\g)$ is defined in (\ref{gama}) and $\cdot$ denotes the Clifford action. 
	
	\subsection{The twisted spinor bundle over a surface}
	
	Let $(M,g)$ be a pseudo-Riemannian surface of signature $(r_1,s_1)$, where $(r_1,s_1)=(0,2)$ if $M$ is Riemannian and $(r_1,s_1)=(1,1)$ if $M$ is Lorentzian. We suppose that $M$ is orientable in the Riemannian case, and orientable in space and in time in the Lorentzian case. Let us consider the trivial vector bundle $E=M\times \R$ endowed with a metric of signature $(r_2,s_2)$ in each fiber, with
	$$(r_1,s_1)+(r_2,s_2)=(1,2).$$ 
	So the metric on $E$ is $- d\nu^2$ if $M$ is Riemannian and $+d\nu^2$ if $M$ is Lorentzian. Let $\widetilde{Q}_M$ and $\widetilde{Q}_E$ be spin structures in $M$ and $E$. Since the orthonormal frame bundle of $E$ is $Q_E=M\times \{1\}$ we have that $\widetilde{Q}_E\longrightarrow M$ is the trivial double covering with fiber $\{1,-1\}$. Let us now define the spinor bundles $\Sigma M$ and $\Sigma E.$ For $k=1,2$ consider the following representations of the spin group
	\begin{eqnarray*}\label{rep_haz_espinoresM_riem}
	\rho_k: Spin(r_k,s_k)&\longrightarrow&Gl(Cl_{r_k,s_k})\\
	x&\longmapsto&\begin{array}[t]{rcc}
	\rho_k(x):Cl_{r_k,s_k}&\longrightarrow &Cl_{r_k,s_k}\nonumber\\
	v &\longmapsto& x\cdot v.
	\end{array} 
	\end{eqnarray*}
	and the spinor bundles over $M$ and $E$ 
	\begin{eqnarray*}\label{haz_espinoresME_riem}
	\Sigma M:= \overset{\sim}{Q}_M\times_{\rho_1 }Cl_{r_1,s_1}& \text{and}&\Sigma E:= \overset{\sim}{Q}_E\times_{\rho_2} Cl_{r_2,s_2}.
	\end{eqnarray*}
	Since the algebras $Cl_{r_1,s_1}\widehat{\otimes} Cl_{r_2,s_2}$ and $Cl_{r_1+r_2,s_1+s_2} $ are isomorphic, the vector bundle  $\Sigma M\otimes \Sigma E$ is isomorphic to
	\begin{eqnarray}\label{haz_prod_tensorial_riem}
	\Sigma: = \left(\overset{\sim}{Q}_M\times_M \overset{\sim}{Q}_E\right) \times_{\rho} Cl_{1,2},
	\end{eqnarray}
	where
	\begin{eqnarray*}\label{rho}
	\rho:Spin(r_1,s_1)\times Spin(r_2,s_2)&\longrightarrow& Gl(Cl_{1,2})\\
	x=(g_1,g_2)&\longmapsto& \begin{array}[t]{rcc}
	\rho(x):Cl_{1,2}&\longrightarrow &Cl_{1,2}\nonumber\\
	v &\longmapsto& g_1 g_2\cdot v
	\end{array}
	\end{eqnarray*}
	(the representation $\rho$ is equivalent to $\rho_1\otimes\rho_2$). We interpret the sub-bundle
	\begin{eqnarray*}\label{haz_prod_tensorial_riem_unit}
		U\Sigma : =\left( \overset{\sim}{Q}_M\times_M \overset{\sim}{Q}_E\right) \times_{\rho} Spin({1,2})&\subset& \Sigma
	\end{eqnarray*}
	as the bundle of unit spinors. From the adjoint representation
	\begin{equation*}\label{rep_clifford_bundle}
	Spin(2)\times Spin(1)\longrightarrow Spin(1,2) \overset{Ad}{\longrightarrow}SO(1,2)\longrightarrow Gl(Cl_{1,2})
	\end{equation*}
	we have as usual the Clifford bundle
	\begin{equation*}\label{clifford_bundle}
	Cl_{\Sigma}:= \overset{\sim}{Q}\times_{Ad} Cl_{1,2}
	\end{equation*}
	and the Clifford action $\cdot:Cl_{\Sigma}\times \Sigma \longrightarrow  \Sigma$.
	
	If $(M,g)$ is an orientable Riemannian (Lorentzian) surface isometrically immersed in $G$ and $E\rightarrow M$ is the normal bundle of $M$ in $G$ then $\Sigma G_{|M}=\Sigma=\Sigma M\otimes \Sigma E$, and setting $\nabla:=\nabla^M\otimes \nabla ^E$ where $\nabla^M$ and $\nabla^E$ are respectively the Levi-Civita and the normal connections we have the Gauss formula
	\begin{equation}\label{gauss_eq_grupo_noinvariante}
	\nabla_X \varphi=-\frac{1}{2}\sum_{j=1}^2\varepsilon_j e_j\cdot B(X,e_j)\cdot \varphi + \nabla_X^{G}\varphi,\ \ \ \varepsilon_j=g(e_j,e_j)=\pm 1
	\end{equation}
	for all $X\in TM,$ where $B$ is the second fundamental form of the immersion and $\varphi$ is a section of $\Sigma G\vert _{M}.$ If $\varphi$ is moreover left-invariant, it follows from \eqref{derivada_grupo} that \eqref{gauss_eq_grupo_noinvariante} is equivalent to 
	\begin{equation}\label{gauss_eq_grupo}
	\nabla_X \varphi=-\frac{1}{2}\sum_{j=1}^2\varepsilon_j e_j\cdot B(X,e_j)\cdot \varphi + \frac{1}{2}\Gamma(X)\cdot \varphi,\ \ \varepsilon_j=g(e_j,e_j),
	\end{equation}
	for all $X\in TM$.
	
	\subsection{Definition and properties of $\langle \langle\cdot ,\cdot \rangle \rangle$} In this section we define an application $\langle \langle\cdot ,\cdot \rangle \rangle:\Sigma\times \Sigma \longrightarrow Cl(\g)$ which will permit us to give an explicit representation formula for a surface immersed in a 3-dimensional Lorentzian Lie group.
	
	Let us consider the involutive anti-automorphism $\tau:Cl(\g)\longrightarrow Cl(\g)$ defined as the linear extension of the transposition of vectors $x_1\cdots x_k \longmapsto x_k\cdots x_1$ and the map
	\begin{eqnarray*}
		\langle \langle \cdot ,\cdot \rangle \rangle: \hspace{.5cm}Cl(\g)\times Cl(\g) &\longrightarrow& Cl(\g)\\
		(\sigma_1,\sigma_2) &\longmapsto& \tau(\sigma_2)\sigma_1.
	\end{eqnarray*}
	If $g\in Spin(\g)$ and $v,w\in Cl(\g)$ we have
	$$\langle \langle g\cdot v ,g\cdot w \rangle \rangle=\tau(g\cdot w)\cdot g\cdot v=\tau(w)\cdot \tau(g)\cdot g\cdot v=\tau(w)\cdot v=\langle \langle v , w \rangle \rangle,$$
	which shows that $\langle \langle \cdot ,\cdot \rangle \rangle$ is $Spin(\g)$-equivariant and induces a map
	\begin{eqnarray}\label{extension_transpuesta}
	\langle \langle \cdot ,\cdot \rangle \rangle:\hspace{.5cm} \Sigma \times \Sigma &\longrightarrow& Cl(\g)\\
	(\varphi,\psi) &\longmapsto& \tau([\psi])[\varphi],\nonumber
	\end{eqnarray}
	where the brackets $[\varphi],[\psi]\in Cl(\g)$ are the coordinates of $\varphi,\psi$ in a spinorial frame. The proofs of following two lemmas can be found in \cite{BRZ}.
	
	\begin{lem}\label{propiedades_transpuesta}
		For all $\varphi, \psi\in \Gamma(\Sigma)$ and $X\in \Gamma(TM\oplus E)$ 
		\begin{eqnarray}\label{tau_1}
		\langle \langle \varphi,\psi \rangle \rangle=\tau \langle \langle \psi, \varphi \rangle \rangle
		\end{eqnarray}
		and 
		\begin{eqnarray}\label{tau_2}
		\langle \langle X\cdot \varphi,\psi \rangle \rangle=\langle \langle \varphi,X\cdot \psi \rangle \rangle.
		\end{eqnarray}
	\end{lem}
	
	\begin{lem}\label{compatibilidad_transpuesta}
		The connection $\nabla$ is compatible with the product $\langle \langle \cdot ,\cdot \rangle \rangle$, that is,
		\begin{eqnarray*}
			\partial_{X}\langle \langle \varphi,\psi \rangle \rangle=\langle \langle \nabla_X\varphi,\psi \rangle \rangle+\langle \langle \varphi,\nabla_X \psi \rangle \rangle
		\end{eqnarray*}
		for all $\varphi,\psi \in \Gamma(\Sigma)$ and $X\in TM$.
	\end{lem}
	
	\subsection{Representation of a surface in a $3$-dimensional Lorentzian Lie group}\label{theorem_grupos}
	\label{inmersion_en_grupos}
	
	Let $G$ be a $3$-dimensional Lorentzian Lie group endowed with a left-invariant metric $\langle\cdot , \cdot \rangle$ and $(M,g)$ an oriented pseudo-Riemmanian surface of signature $(r_1,s_1)$. We consider the trivial bundle $E=M\times \R$ with metric $-d\nu^2$ if $M$ is Riemannian or $+d\nu^2$ if $M$ is Lorentzian, and suppose that a symmetric bilinear form $B:TM\times TM\longrightarrow E$ is given. The following \emph{compatibility conditions} on $M$, $G$ and $B$ will appear to be necessary to state our theorem.
	\begin{enumerate}
		\item \label{compa1} There exists a bundle isomorphism
		\begin{equation*}\label{iso_haces}
		f:TM\oplus E\longrightarrow M\times \mathfrak{g}
		\end{equation*}
		which preserves the metrics; from this isomorphism we define  
		\begin{equation*}\label{gama_TME}
		\Gamma:TM\oplus E\longrightarrow \Lambda^2(TM\oplus E) \subset Cl_{\Sigma}
		\end{equation*}
		such that, for all $X, Y\in \Gamma(TM\oplus E)$,
		\begin{equation*}\label{gama_f}
		f(\Gamma(X)(Y))=\Gamma(f(X))(f(Y)),
		\end{equation*}
		where
		in the right hand side of the equation, $\Gamma$ is the application defined in \eqref{gama}. Furthermore, we say that a section $Z$ of $TM\oplus E$ is left-invariant if $f(Z):M\longrightarrow \g$ is a constant map.
		\item \label{compa2} If $\nabla:=\nabla^M\oplus\nabla^E$ is the sum of the Levi-Civita connection on $M$ and the trivial connection on $E$, we assume that the covariant derivative of a left-invariant section $Z\in \Gamma(TM\oplus E)$ is
		\begin{equation}\label{derivada_TME}
		\nabla_{X}Z=\Gamma(X)(Z)-B(X,Z^T)+B^*(X,Z^N),
		\end{equation} 
		for all $X\in TM$, where $Z=Z^T+Z^N\in TM\oplus E$ and $B^*:TM\times E \longrightarrow TM$ is the bilinear operator that satisfies
		\begin{equation*}\label{segunda_forma_fund}
		\langle B(X,Y),N\rangle=\langle Y,B^*(X,N)\rangle,
		\end{equation*}
		for all $X,Y\in \Gamma(TM)$ and $N\in \Gamma(E)$.
	\end{enumerate}
	Now, we enunciate the main theorem of that section whose proof in a slightly different context can be found in \cite{BRZ}.
	\begin{thm}\label{thm main result}
		We suppose that $M$ is simply connected. The following statements are equivalent:
		\begin{enumerate}
			\item There exists a section $\varphi\in\Gamma(U\Sigma)$ such that
			\begin{eqnarray}\label{gauss_grupos}
			\nabla_X\varphi=-\frac{1}{2}\sum_{j=1}^2\varepsilon_j e_j\cdot B(X,e_j)\cdot\varphi+\frac{1}{2}\Gamma(X)\cdot\varphi, & \varepsilon_j=g(e_j,e_j)  
			\end{eqnarray}
			for all $X\in TM$.
			\item There exists an isometric immersion $F:\ M\rightarrow G$ with normal bundle $E$ and second fundamental form $B.$
		\end{enumerate}
		Precisely, if $\varphi$ is a solution of equation $\eqref{gauss_grupos}$, replacing $\varphi$ by $\varphi\cdot a$ for some $a\in Spin(\g)$ if necessary, we consider the $\g$-valued 1-form given by
		\begin{equation}\label{rep_exp_grupos}
		\xi(X)=\langle \langle X\cdot \varphi, \varphi \rangle \rangle ,
		\end{equation}
		for all $X\in TM$ and the formula $F=\int \xi$ defines an isometric immersion $F:M\longrightarrow G$ with normal bundle $E$ and second fundamental form $B$. Here $\int$ denotes the Darboux integral (see \cite[Pag. 165]{PMal}), i.e., $F=\int \xi: M\longrightarrow G$ is such that $F^*\omega=\xi$ for $\omega\in \Omega^1(G,\g)$ the Maurer-Cartan form defined in \eqref{maurer_cartan}. Reciprocally, an isometric immersion $M\longrightarrow G$ with normal bundle $E$ and second fundamental form $B$ can be written in that way.
	\end{thm}

	The explicit representation formula $F=\int \xi$ is a generalized Weierstrass formula for Lie groups.
	
	\begin{rem}\label{compatibilidad_coordenadas}
		Let us write the compatibility conditions using a frame $(e_1^o,e_2^o,e_3^o)$ of $\g$ such that $\langle e_i^o,e_j^o\rangle=0$ if $i\neq j$ and $\langle  e_1^o,e_1^o\rangle=\langle e_2^o,e_2^o\rangle=1=-\langle e_3^o,e_3^o\rangle$. We denote by $\Gamma_{ij}^k\in \R$, $1\leq i,j,k\leq 3$ the constants that satisfy
		\begin{eqnarray*}
			\Gamma(e_i^o)(e_j^o)=\sum_{k=1}^{3}\varepsilon_k\Gamma_{ij}^ke_k^o,& \varepsilon_k=\langle e_k^o,e_k^o \rangle=\pm 1,
		\end{eqnarray*}
		and let $N$ be a unit section of $E$: $\langle N,N\rangle=\varepsilon$, where $\varepsilon=-1$ if $M$ is a Riemannian surface and $\varepsilon=+1$ if $M$ is a Lorentzian surface. For $i\in \{1,2,3\}$ we choose $\underline{e}_i\in \Gamma(TM\oplus E)$ such that $f(\underline{e}_i)=e_i^o,$ and consider $\nu_i \in C^{\infty}(M)$ and $T_i \in \Gamma(TM)$ so that $\underline{e}_i=T_i+\nu_i N$; since $f$ is an isometry $(\underline{e}_1,\underline{e}_2,\underline{e}_3)$ is an orthonormal frame of $TM\oplus E$ and we have
		\begin{equation}\label{ortonormalidad}
		\langle T_i,T_j\rangle +\varepsilon \nu_i \nu_j=\varepsilon_i\delta_{ij},
		\end{equation}
		for all $i,j\in \{1,2,3\}$. With this notation, the equation
		\eqref{derivada_TME} is equivalent to
		\begin{eqnarray}
		\nabla_X T_j&=& \sum_{1\leq i,k\leq 3}\varepsilon_i \varepsilon_k \Gamma_{ij}^k \langle X, T_i \rangle T_k +\nu_j S(X)\label{derivada_tj}\\
		d\nu_j(X)&=& \sum_{1\leq i,k\leq 3}\varepsilon_i\varepsilon_k \Gamma_{ij}^k \langle X, T_i \rangle\nu_k -h(X,T_j)\label{derivada_nuj},\ \  1\leq j\leq 3,
		\end{eqnarray}
		where $S(X):=B^*(X, N)$ and $h(X,T_j):=\langle B(X,T_j), N\rangle$. Reciprocally, if there exist vector fields $T_i\in \Gamma(TM)$ and functions $\nu_i\in C^{\infty}(M)$  for $i\in \{1,2,3\}$ such that \eqref{ortonormalidad}-\eqref{derivada_nuj} are satisfied, then $f:TM\oplus E \longrightarrow M\times \g$ given by $f(\underline{e}_i)=e_i^o$ for $\underline{e}_i=T_i+\nu_i N$ is a bundle isomorphism that preserves the metrics and such that \eqref{derivada_TME} is satisfied.
	\end{rem}
	
	\section{Intrinsic version of the immersion theorem in Lie groups}\label{version_intrinseca}
	In this section we rewrite Theorem \ref{thm main result} using only the usual spinor bundle $\Sigma M.$
	\subsection{Immersion of a Riemannian surface}
	Let $M$ be a simply connected Riemannian surface and let $E=M\times \R$ be endowed with the metric $-d\nu^2$ in each fiber. For a symmetric bilinear form $B:TM\times TM \longrightarrow E$ and a section $N$ of $E$ such that $\langle N ,N \rangle=-1,$ we define the symmetric operator $S:TM\longrightarrow TM$ so that $\langle B(X,Y),N \rangle=\langle S(X),Y\rangle$ for all $X,Y\in TM$; explicitly, for all $X\in TM$ 
	$$S(X)=\sum_{j}e_j\langle B(X,e_j),N \rangle.$$
	We assume that there exist vector fields $T_i\in \Gamma(TM)$ and functions $\nu_i \in C^{\infty}(M)$, $i = 1, 2, 3,$ satisfying the compatibility conditions \eqref{ortonormalidad}-\eqref{derivada_nuj}. Let us set for all $X\in TM$
	\begin{equation*}\label{gamma1}
	\Gamma_1(X)=\sum_{i=1}^3 \varepsilon_i \langle X,T_i\rangle \sum_{j<k}\varepsilon_j \varepsilon_k \Gamma_{ij}^k\frac{1}{2}(T_j\cdot T_k-T_k\cdot T_j)  
	\end{equation*}
	and
	\begin{equation*}\label{gamma2}
	\Gamma_2(X)= \sum_{i=1}^3 \varepsilon_i \langle X,T_i\rangle \sum_{j<k}\varepsilon_j \varepsilon_k\Gamma_{ij}^k  (\nu_j T_k-\nu_k T_j).\end{equation*}
	
	The decomposition $Cl_{1,2}=Cl^0_{1,2}\oplus Cl^1_{1,2}$ induces a splitting $\Sigma=\Sigma_0\oplus\Sigma_1$ of the spinor bundle, and there is a natural identification
	\begin{eqnarray}\label{id_riemanniana_rev*}	
	\Sigma 
	M&\rightarrow&\Sigma_0\\
	\psi&\mapsto&\psi^*\nonumber
	\end{eqnarray}
	such that
	\begin{equation} \label{id_estrella_1_riem}
	(\nabla_X\psi)^*=\nabla_X\psi^*,\hspace{.3cm}(X\cdot\psi)^*=iN\cdot X\cdot\psi^*,\hspace{.3cm}|\psi^+|^2-|\psi^-|^2=\langle\langle\psi^*,\psi^*\rangle\rangle
	\end{equation}
	for all $X\in TM$ (details are given in Appendix \ref{spinor_bundle_identification}).

	\begin{thm} \label{riemanniana}The following statements are equivalent:
		\begin{enumerate}
			\item There exists an isometric immersion of $M$ into $G$ with shape operator $S$.
			\item There exists $\psi\in \Gamma(\Sigma M)$ solution of
			\begin{equation}\label{inmersionRG} \nabla_{_X}\psi=\frac{i}{2}S(X)\cdot \psi+\frac{1}{2}\Gamma_1(X)\cdot {\psi} +\frac{i}{2}\Gamma_2(X)\cdot \psi \end{equation}
			for all $X\in TM$ and such that $\vert \psi^{+}\vert ^2-\vert \psi^{-}\vert ^2=1$.
		\end{enumerate}
	\end{thm}
	The proof of theorem comes after the following lemmas.

	Recall the notation of Remark \ref{compatibilidad_coordenadas}: $(\underline{e}_1,\underline{e}_2,\underline{e}_3)$ is an orthonormal basis of invariant vector fields in $TM\oplus E$, and for $k={1,2,3}$, $T_k\in \Gamma(TM)$ and $\nu_k\in C^{\infty}(M)$ are such that $\underline{e}_k=T_k+\nu _k N$.
	The proof of the following lemma can be found in a different context in \cite{BRZ}.
	\begin{lem}\label{gama_bivector_coord} For all $X\in TM,$
		\begin{equation*}
		\Gamma(X)=\sum_{i}\varepsilon_i\langle X,T_i\rangle\sum_{j<k}\varepsilon_j\varepsilon_k\Gamma_{ij}^k\left(\frac{1}{2}(T_j\cdot T_k-T_k\cdot T_j)+(\nu_kT_j-\nu_jT_k)\cdot N\right),
		\end{equation*}
		where $\varepsilon_i=\langle \underline{e}_i,\underline{e}_i\rangle=\pm 1$.
	\end{lem}
	
	After some calculations we have from the properties in \eqref{id_estrella_1_riem} the following lemma (details can be found in \cite{TBZJ}).
	\begin{lem}\label{gama_bivector_coord_prop}  For $T_j, T_k \in \Gamma(TM)$, $\nu_j, \nu_k\in C^{\infty}(M)$, $N$ a unit section of $E$ and $\psi$ a spinor field of $\Sigma M$ the following identities are true:
		\begin{enumerate}
			\item \begin{equation*}
			(T_j\cdot T_k -T_k\cdot T_j) \cdot\psi^*=((T_j\cdot T_k -T_k\cdot T_j) \cdot\psi)^*,
			\end{equation*}
			\item \begin{equation*}(\nu_k T_j-\nu_jT_k)\cdot N \cdot \psi^*=(i(\nu_k T_j-\nu_jT_k)\cdot \psi)^*,\end{equation*}
			\item \begin{equation*}-\frac{1}{2}\sum_{j=1}^{2}e_j\cdot B(X,e_j)\cdot \psi^*=\left(\frac{i}{2}S(X)\cdot \psi\right)^*.\end{equation*}
		\end{enumerate}
	\end{lem}
	
	\begin{proof}[Proof of Theorem \ref{riemanniana}.]
		
		Let $\psi\in \Gamma(\Sigma M)$ be a solution of \eqref{inmersionRG} such that $\vert \psi^+\vert^2-\vert \psi^-\vert^2=1$. If we apply the isomorphism \eqref{id_riemanniana_rev*} to both sides of that equation we deduce from Lemmas \ref{gama_bivector_coord} and \ref{gama_bivector_coord_prop}  that
		\begin{equation}(\nabla_{_X}\psi)^*=(\frac{i}{2}S(X)\cdot \psi)^*+(\frac{1}{2}\Gamma_1(X)\cdot {\psi} +\frac{i}{2}\Gamma_2(X)\cdot \psi)^* \end{equation}
		is equivalent to
		\begin{equation}\label{interna_inmersion_riemanniana}\nabla_{_X}\psi^*= -\frac{1}{2}\sum_{j=1}^{2}e_j\cdot B(X,e_j)\cdot \psi^*+\frac{1}{2}\Gamma(X) \cdot \psi^*, \end{equation}
		where $\varphi:=\psi^*$ is in $\Gamma(U\Sigma);$ Equation \eqref{interna_inmersion_riemanniana} is equivalent to \eqref{gauss_grupos} and the conclusion follows from Theorem \ref{thm main result}.
	\end{proof}
	The following corollary gives a new proof of the spinorial characterization of the immersions in
	$\R^{1,2}$ obtained in \cite{BP}. See \cite{LR} for a result using two spinor fields.
	\begin{cor}\label{coro_riemanniana} 
		If $G=\R_1^3$ the following statements are equivalent:
		\begin{enumerate}
			\item There exists an isometric immersion of $M$ into $\R_1^3$ with shape operator $S$.
			\item There exists $\psi\in \Gamma(\Sigma M)$ solution of 
			\begin{equation}\label{inmersion_R12} \nabla_{_X}\psi=\frac{i}{2}S(X)\cdot \psi,\end{equation} 
			for all $X\in TM$ and such that $\vert \psi^{+}\vert ^2-\vert \psi^{-}\vert ^2=1$ .
		\end{enumerate}
	\end{cor}
	\begin{proof}
		In this case $\Gamma_{ij}^k=0$ for all $i,j,k\in\{1,2,3\}$ and then $\Gamma_1(X)=0=\Gamma_2(X)$ .
	\end{proof}
	\begin{rem}\label{formula_explicita_componentes}
		If $\psi=\psi^++\psi^-$ is as in Theorem \ref{riemanniana} then the corresponding $1$-form $\xi(X)=\langle \langle X\cdot \varphi, \varphi\rangle \rangle$ of Theorem \ref{thm main result} reads 
		\begin{equation}\label{explicita}
		\xi(X)=2iI \Im m\langle X\cdot \psi^- ,\psi^{+}\rangle+ J(\langle X\cdot \psi^+, \alpha(\psi^+) \rangle-\langle X\cdot \psi^-, \alpha(\psi^-) \rangle),
		\end{equation}
		where $\alpha:\Sigma M \longrightarrow \Sigma M$ is a quaternionic structure; the immersion is thus explicitly given in terms of $\psi;$ the proof may be found in \cite[Proposition 3.4.13]{TBZJ}.
	\end{rem}
	
	\subsection{Immersion of a Lorentzian surface}
	Let $(M,g)$ be a simply connected Lorentzian surface and let $E=M\times \R$ be endowed with the metric $+d\nu^2$ in each fiber. Let $B:TM\times TM \longrightarrow E$ be a bilinear symmetric form and consider its associated symmetric operator $S:TM\longrightarrow TM$ given by
	$$S(X)=\sum_{j=1}^2 \varepsilon_j\langle B(X,e_j),N \rangle e_j,\ \varepsilon_j=\langle e_j,e_j\rangle=\pm 1,$$ 
	for all $X\in TM$ and $N$ a fixed unit section of $E$. We assume that there exist vector fields $T_i\in \Gamma(TM)$ and functions $\nu_i \in C^{\infty}(M)$, $i = 1, 2, 3,$ satisfying the compatibility conditions \eqref{ortonormalidad}-\eqref{derivada_nuj}. Let us set for all $X\in TM$
	\begin{equation}\label{gamma_tilda}
	\overset{\sim}{\Gamma}(X)=\sum_{i}\varepsilon_i\langle X,T_i\rangle\sum_{j<k}\varepsilon_j\varepsilon_k\Gamma_{ij}^k\left(\frac{1}{2}(T_j\cdot T_k-T_k\cdot T_j)+(\nu_kT_j-,\nu_jT_k)\right).
	\end{equation} 
	
	The decomposition $Cl_{1,2}=Cl^0_{1,2}\oplus Cl^1_{1,2}$ induces a splitting $\Sigma=\Sigma_0\oplus\Sigma_1$ of the spinor bundle, and there is a natural identification
	\begin{eqnarray}
	\Sigma M&\rightarrow&\Sigma_0\label{identif_lorentziana}\\
	\psi&\mapsto&\psi^*\nonumber
	\end{eqnarray}
	such that
	\begin{equation}\label{prop_identif_lorentziana}
	(\nabla_X\psi)^*=\nabla_X\psi^*,\hspace{.3cm}(X\cdot\psi)^*=X\cdot N \cdot \psi^*,\hspace{.3cm}|\psi^+|^2-|\psi^-|^2=\langle\langle\psi^*,\psi^*\rangle\rangle
	\end{equation}
	for all $X\in TM$ (details are given in Appendix \ref{spinor_bundle_identification}).
	
	\begin{thm}\label{lorentziana} The following statements are equivalent:
		\begin{enumerate}
			\item There exists an isometric immersion of $M$ into $G$ with shape operator $S$.
			\item There exists $\psi\in \Gamma(\Sigma M)$ solution of
			\begin{equation}\label{lorentziana_grupo} \nabla_{_X}\psi=-\frac{1}{2}S(X)\cdot \psi+\frac{1}{2}\overset{\sim}{\Gamma}(X)\cdot {\psi} , \end{equation}
			for all $X\in TM$ and such that $\vert \psi^+ \vert^2-\vert \psi^- \vert^2=1$.
		\end{enumerate}
	\end{thm}
	\begin{proof}
		Let $\psi$ be a section of $\Sigma M$ with $\vert \psi^+\vert^2- \vert \psi^-\vert^2=1$ that satisfies \eqref{lorentziana_grupo}. If we apply the isomorphism \eqref{identif_lorentziana} to that equation it follows from the properties \eqref{prop_identif_lorentziana} that $(\nabla_{_X}\psi)^*=\nabla_X\psi^*$, 
		$(-\frac{1}{2}S(X)\cdot \psi)^*
		=-\frac{1}{2}\sum_{j=1}^2 \varepsilon_j  e_j \cdot B(X,e_j)\cdot \psi^*$,
		and, if $\Gamma(X)$ and $\overset{\sim}{\Gamma}(X)$ are given by Lemma \ref{gama_bivector_coord} and equation \eqref{gamma_tilda} respectively, that
		\begin{equation*}
		\left(\frac{1}{2}\overset{\sim}{\Gamma}(X)\cdot \psi\right)^*=\frac{1}{2}{\Gamma}(X)\cdot \psi^*.
		\end{equation*}
		Therefore, the following equation 
		\begin{equation}\label{gauss_estrella_lorentz}
		(\nabla_{_X}\psi)^*=(-\frac{1}{2}S(X)\cdot \psi)^*+(\frac{1}{2}\overset{\sim}{\Gamma}(X)\cdot {\psi})^*,
		\end{equation}
		is equivalent to
		\begin{equation}\label{interna_inmersion_lorentziana}
		\nabla_{_X}\psi^*=-\frac{1}{2}\sum_{j=1}^2 \varepsilon_j  e_j \cdot B(X,e_j)\cdot \psi^*+\frac{1}{2}{\Gamma}(X)\cdot {\psi}^*,
		\end{equation}
		so if we set $\varphi:=\psi^*$ the Equation \eqref{interna_inmersion_lorentziana} is equivalent to \eqref{gauss_grupos} and the result is a consequence of Theorem \ref{thm main result}.
	\end{proof}
	
	\begin{cor}
		If $M$ is a simply connected Lorentzian surface and $G=\R_1^3$ then the following statements are equivalent:
		\begin{enumerate}
			\item There exists an isometric immersion of $M$ in $\R_1^3$ with shape operator $S$.
			\item There exists $\psi\in \Gamma(\Sigma M)$ solution of
			\begin{equation*}
			\nabla_{_X}\psi=-\frac{1}{2}S(X)\cdot \psi,
			\end{equation*} 
			for all $X\in TM$ and such that $\vert\psi^+\vert^2- \vert\psi^-\vert^2=1$. 
		\end{enumerate}
	\end{cor}
	See \cite{La} and \cite{LR} for a similar result involving two spinor fields.
	\begin{proof}
		In that case $\Gamma_{ij}^k=0$ for all $i,j,k\in \{1,2,3\}$ and then $\overset{\sim}{\Gamma}(X)=0$ for all $X\in TM$.
	\end{proof}
	
	\section{Representation of a surface in Lorentzian products}\label{products}
	In this section we characterize the immersions of an orientable Riemannian surface $(M,g)$ in the symmetric Lorentzian products $\mathbb{H}^2\times \R_{-}$, $\R\times \mathbb{S}_1^2$ and $\R\times \mathbb{H}_1^2$ which are  locally isometric to a Lie group with one of the following Lie algebras:
	\begin{eqnarray*}
		\mathfrak{a}:\begin{array}{cclc}
			\left[e_1,e_2\right]&=&\alpha e_2\\
			\left[e_1,e_3\right]&=&0\\
			\left[e_2,e_3\right]&=
			&0
		\end{array}& 
		\mathfrak{b}:\begin{array}{cclc}
			\left[e_1,e_2\right]&=&0&\\
			\left[e_1,e_3\right]&=&\alpha e_1  \\
			\left[e_2,e_3\right]&=& 0
		\end{array}& \mathfrak{c}:\begin{array}{cclc}
			\left[e_1,e_2\right]&=&0&\\
			\left[e_1,e_3\right]&=& \delta e_3\\
			\left[e_2,e_3\right]&=&0,\\
		\end{array}
	\end{eqnarray*}
	%\vspace{0pt}
	where $\alpha, \delta \neq 0,$ $\langle  e_1,e_1 \rangle=\langle e_2,e_2 \rangle=1=-\langle e_3,e_3 \rangle$ and $\langle e_i,e_j \rangle=0$ for $i\neq j$. By the Koszul formula, the only non-zero components for the covariant derivative $\overline{\nabla}$ are, in each case,
	\begin{eqnarray}
	\begin{array}{ll}
	\mathfrak{a}:&\overline{\nabla}_{e_2}e_1=-\alpha e_2 \\
	&\overline{\nabla}_{e_2}e_2=\alpha e_1,
	\end{array}\label{compg7}&
	\begin{array}{ll}
	\mathfrak{b}:&\overline{\nabla}_{e_1}e_1=\alpha e_3 \\ &\overline{\nabla}_{e_1}e_3=\alpha e_1,
	\end{array}
	\label{compg5}&
	\begin{array}{ll}
	\mathfrak{c}:&\overline{\nabla}_{e_3}e_1=-\delta e_3 \\ &\overline{\nabla}_{e_3}e_3=-\delta e_1.
	\end{array}
	\label{compg6}
	\end{eqnarray}
	We assume the following compatibility conditions (see Remark \ref{compatibilidad_coordenadas}): there exist tangent vector fields $T_i\in \Gamma(TM)$ and functions $\nu_i\in C^{\infty}(M)$ such that 
	$$ \langle T_i,T_j\rangle - \nu_i \nu_j=\varepsilon_i\delta_{ij},\hspace{.5cm}1\leq i,j\leq 3$$
	and, for $1\leq j\leq 3,$
	$$\nabla_X T_j= \sum_{1\leq i,k\leq 3}\varepsilon_i \varepsilon_k \Gamma_{ij}^k \langle X, T_i \rangle T_k +\nu_j S(X)$$
	and
	$$d\nu_j(X)=\sum_{1\leq i,k\leq 3}\varepsilon_i\varepsilon_k \Gamma_{ij}^k \langle X, T_i \rangle\nu_k -h(X,T_j),$$
	where $S(X)=B^*(X, N)$, $h(X,T_j)=\langle B(X,T_j), N\rangle$ and $\varepsilon_j=\pm 1$.
	\begin{thm}\label{H2xR}
		If $M$ is simply connected then the following statements are equivalent:
		\begin{enumerate}
			\item There exists $\psi \in \Gamma(\Sigma M)$ solution of
			\begin{equation}\label{intrinseca_H2R}
			\nabla_{_X}\psi=\frac{i}{2}S(X)\cdot \psi -\frac{\alpha}{2} \langle X,T_2\rangle (iT_3+\nu_3)\cdot \omega \cdot \psi
			\end{equation}
			for all $X\in TM$ and such that $\vert \psi^+\vert^2-\vert \psi^-\vert^2=1$. Here $\omega=\epsilon_1\cdot \epsilon_2$ where $(\epsilon_1,\epsilon_2)$ is a positively oriented orthonormal basis of $M$.
			\item There exists an isometric immersion of $M$ into $\mathbb{H}^2\times \R_{-}$ with shape operator $S$.
		\end{enumerate}
	\end{thm}
	\begin{proof}
		We prove that the existence of a solution of equation \eqref{intrinseca_H2R} is equivalent to the existence of a solution of equation \eqref{gauss_grupos} and then use Theorem \ref{thm main result} to obtain the result. It follows from Lemma \ref{lemaB1} that the bivector representing $\Gamma(X)$ is
		\begin{equation*}
		\Gamma(X)=\frac{1}{2}(\underline{e}_1\cdot\Gamma(X)(\underline{e}_1)+\underline{e}_2\cdot\Gamma(X)(\underline{e}_2)-\underline{e}_3\cdot\Gamma(X)(\underline{e}_3)),
		\end{equation*}
		where $(\underline{e}_1,\underline{e}_2,\underline{e}_3)$ is as in Remark \ref{compatibilidad_coordenadas}. By \eqref{compg7} we have that
		\begin{eqnarray*}
			\Gamma(X)(\underline{e}_1)=-\alpha \langle X,\underline{e}_2\rangle \underline{e}_2,&
			\Gamma(X)(\underline{e}_2)=\alpha \langle X,\underline{e}_2\rangle \underline{e}_1 & \text{and}\ \ \  \Gamma(X)(\underline{e}_3)=0,
		\end{eqnarray*}
		and then $
		\Gamma(X)=-\alpha \langle X,\underline{e}_2\rangle \underline{e}_1\cdot \underline{e}_2.$ Since $\underline{e}_k=T_k+\nu_k N$ and $\underline{e}_1\cdot \underline{e}_2 \cdot \underline{e}_3=\omega \cdot N$ it follows that
		\begin{eqnarray*}
			\Gamma(X)&=&-\alpha \langle X,\underline{e}_2\rangle \underline{e}_1\cdot \underline{e}_2\\
			&=&-\alpha \langle X,\underline{e}_2\rangle \underline{e}_1\cdot \underline{e}_2\cdot \underline{e}_3\cdot \underline{e}_3\\
			&=&-\alpha \langle X,T_2\rangle \omega \cdot N\cdot (T_3+\nu_3 N)\\
			&=&-\alpha \langle X,T_2\rangle(-N \cdot T_3+\nu_3 )\cdot \omega .
		\end{eqnarray*}

		We now apply to the equation \eqref{intrinseca_H2R} the isomorphism \eqref{id_riemanniana_rev*} to obtain
		\begin{eqnarray*}
			-\frac{1}{2}\sum_{j=1}^{2}e_j\cdot B(X,e_j)\cdot \psi ^*=\left(\frac{i}{2}S(X)\cdot \psi \right)^*
		\end{eqnarray*}
		and
		\begin{eqnarray*}
			\frac{1}{2}\Gamma(X)\cdot \psi^*=\left( -\frac{1}{2}\alpha \langle X,T_2\rangle(i T_3+\nu_3 )\cdot \omega \cdot \psi\right)^*. \end{eqnarray*}
		Therefore, if $\varphi:=\psi^*$, the equation 
		\begin{equation*}
		\nabla_{_X} \varphi =	-\frac{1}{2}\sum_{j=1}^{2}e_j\cdot B(X,e_j)\cdot \varphi +\frac{1}{2}\Gamma(X)\cdot \varphi
		\end{equation*}
		is equivalent to \eqref{intrinseca_H2R}.
	\end{proof}
	
	The proofs of Theorems \ref{RxS_12} and \ref{RxH12} below are similar to the proof of Theorem \ref{H2xR} so we omit them.
	
	\begin{thm}\label{RxS_12}
		If $M$ is simply connected, the following statements are equivalent:
		\begin{enumerate}
			\item There exists $\psi \in \Gamma(\Sigma M)$ solution of
			\begin{equation}\label{intrinseca_RS12}
			\nabla_{_X}\psi=\frac{i}{2}S(X)\cdot \psi +\frac{\alpha}{2} \langle X,T_1\rangle (iT_2+\nu_2)\cdot \omega \cdot \psi
			\end{equation}
			for all $X\in TM$ and such that $\vert \psi^+\vert^2-\vert \psi^-\vert^2=1$.
			\item There exists an isometric immersion of $M$ into $\R\times\mathbb{S}_1^2 $ with shape operator $S$.
		\end{enumerate}
	\end{thm}
	\begin{thm}\label{RxH12}
		If $M$ is simply connected then the following statements are equivalent:
		\begin{enumerate}
			\item There exists $\psi \in \Gamma(\Sigma M)$ solution of
			\begin{equation}\label{intrinseca_RH12}
			\nabla_{_X}\psi=\frac{i}{2}S(X)\cdot \psi +\frac{\alpha}{2} \langle X,T_3\rangle (iT_2+\nu_2)\cdot \omega \cdot \psi
			\end{equation}
			for all $X\in TM$ and such that $\vert \psi^+\vert^2-\vert \psi^-\vert^2=1$.
			\item There exists an isometric immersion of $M$ into $\R\times\mathbb{H}_1^2 $ with shape operator $S$.
		\end{enumerate}
	\end{thm}
	\section{Representation of a surface in $\mathbb{S}_1^3$ and $\mathbb{H}_1^3$ }\label{sitter_antisitter}
	\subsection{Representation of a surface in $\mathbb{S}_1^3$}
	The 3-dimensional de Sitter space is the Lorentzian hypersurface  of $\R^{1,3}$ with constant sectional curvature $1$ given by the quadric
	\begin{eqnarray*}
	\mathbb{S}_1^3=\{(x_1,x_2,x_3,x_4)\in \R^{1,3}:-x_1^2+x_2^2+x_3^2+x_4^2=1\}.
	\end{eqnarray*}

	We suppose that $M$ is a simply connected Riemannian surface and $E=M\times \R$ is the trivial bundle with metric $-d\nu^2$ in the fibers. Let us assume that a symmetric bilinear form $B:TM\times TM \longrightarrow E$ is given.
	
	We consider the splitting $\R^{1,3}=\R^{1,2}\oplus \R e_4$ where $e_4$ is the last element of the standard orthonormal basis of $\R^{1,3},$ the corresponding application
	\begin{eqnarray*}
		Spin(0,2)\times Spin(1,0)&\longrightarrow& Spin(1,2)\subset Spin(1,3)\\
		(g_1,g_2)&\longmapsto& g_1 g_2,
	\end{eqnarray*}
	and the representation
	\begin{eqnarray*}
		\rho:Spin(0,2)\times Spin(1,0)&\longrightarrow& Gl(Cl_{1,3})\\
		(g_1,g_2)&\longmapsto& \begin{array}[t]{rcc}
			\rho(g_1,g_2):Cl_{1,3}&\longrightarrow &Cl_{1,3}\nonumber\\
			v &\longmapsto& g_1 g_2\cdot v.
		\end{array}
	\end{eqnarray*}
	
	For $\overset{\sim}{Q}=\overset{\sim}{Q}_M\times_M \overset{\sim}{Q}_E$ the product of spin structures, we define the bundles
	\begin{equation}\label{haz_grande_superficie}
	\Sigma:=\overset{\sim}{Q}\times_{\rho}Cl_{1,3},\hspace{.3cm}U\Sigma :=\overset{\sim}{Q}\times_{\rho}Spin{({1,3})}\subset \Sigma \hspace{.3cm}\mbox{and}\hspace{.3cm}Cl_{\Sigma}:=\overset{\sim}{Q}\times_{Ad}Cl_{1,3},
	\end{equation}
	and the element $\nu=[\overset{\sim}{s},e_4]\in Cl_{\Sigma}$ (it is well-defined, independently of the choice of $\overset{\sim}{s}\in\overset{\sim}{Q},$ since by construction $e_4$ is invariant by the action of $Spin(1,2)$). 
	\begin{thm}\label{inmersion_sitter}
		The following statements are equivalent:
		\begin{enumerate}
			\item There exists $\varphi \in \Gamma( U\Sigma) $ solution of
			\begin{equation}\label{ec_sitter}
			\nabla_X\varphi =
			-\frac{1}{2}\sum_{j=1}^{2}e_j\cdot B(X,e_j)\cdot \varphi+\frac{1}{2}X\cdot \nu \cdot \varphi
			\end{equation}
			for all $X\in TM$.
			\item There exists an isometric immersion $F$ of $M$ into $\mathbb{S}_{1}^3$ with second fundamental  form $B$.
		\end{enumerate}
		Moreover, $F:M\longrightarrow \mathbb{S}_{1}^3$ is given by \begin{equation}
		\label{explicita_siter} F=\langle \langle\nu \cdot \varphi, \varphi \rangle\rangle.
		\end{equation}
	\end{thm}
	\begin{proof}
		We first prove $1\implies 2.$ Let us suppose that $\varphi \in \Gamma(U \Sigma) $ is a solution of \eqref{ec_sitter} and observe that $F$ defined by \eqref{explicita_siter} takes values in $\mathbb{S}_{1}^3$: by definition, $F=\tau[\varphi][\nu][\varphi]$ and since $[\varphi]\in Spin(1,3)$ and $[\nu]\in \R^{1,3}$ we have that $F=Ad([\varphi]^{-1})([\nu])$. Since $Ad([\varphi]^{-1})\in SO(1,3)$ and $[\nu]$ has norm $1$ we conclude that $Ad([\varphi]^{-1})([\nu])\in \mathbb{S}_{1}^3$. The rest of the proof rely on the following results, the proof of the first one can be found in \cite[Lemma 8.1]{BLR} and we omit it here: 
		\begin{lem}\label{lemadFX_Sitter}
			If $\varphi \in \Gamma(U\Sigma) $ satisfies equation \eqref{ec_sitter} and $F:M\longrightarrow \mathbb{S}_1^3$ is defined by \eqref{explicita_siter} then for all $X\in TM$
			\begin{eqnarray*}
				dF(X)=\langle \langle X \cdot \varphi, \varphi \rangle\rangle.
			\end{eqnarray*}
		\end{lem}
		\begin{prop}\label{isometria_haznormal_sitter}The following statements are true:
			\begin{enumerate}
				\item $F:M\longrightarrow \mathbb{S}_1^3$ is an isometry.
				\item The application
				\begin{eqnarray*}
					\Phi_E:\ 	E&\longrightarrow &T\mathbb{S}_1^3\\
					\ X\in E_m&\longmapsto&(F(m),\langle\langle X\cdot \varphi, \varphi \rangle\rangle)
				\end{eqnarray*}
				is an isomorphism between $E$ and the normal bundle of the immersion $F$ which preserves the metrics and identifies $B$ with the second fundamental form of $F$.  
			\end{enumerate}
		\end{prop}
		\begin{proof}
			For simplicity let us set $\xi(X):=\langle\langle X\cdot \varphi, \varphi \rangle\rangle$ for all $X\in TM\oplus E$, then, for all $X, Y\in TM\oplus E$
			\begin{eqnarray*}
				\langle \xi(X), \xi(Y) \rangle&=&- \frac{1}{2}(\xi(X)\xi(Y)+\xi(Y)\xi(X)) \label{isometria}
				\\
				&=& -\frac{1}{2}(\tau[\varphi][X][\varphi]\tau[\varphi][Y][\varphi]+\tau[\varphi][Y][\varphi]\tau[\varphi][X][\varphi]) \nonumber\\
				&=&-\frac{1}{2}\tau[\varphi]([X][Y]+[Y][X])[\varphi] \nonumber\\
				&=&\tau[\varphi]\langle X, Y \rangle[\varphi]\nonumber\\
				&=&\langle X, Y \rangle, \nonumber
			\end{eqnarray*}
			which proves that $F$ and $\Phi_E$ are isometries. Let us see now that $B$ is the second fundamental form of the immersion, in other words, if $NM$ is the normal bundle of $F$ and $B^F:TM\times TM\longrightarrow NM$ is its second fundamental, then
			\begin{eqnarray*}
				B^F(X,Y)=\langle \langle B(X,Y)\cdot \varphi, \varphi \rangle \rangle.
			\end{eqnarray*}
			Let $X,Y$ be sections of $TM$ such that $\nabla X=0=\nabla Y$ at a point; by definition
			\begin{equation*} 
			B^F(X,Y)=\{\partial_X(dF(Y))\}^N=(\partial_X\langle \langle Y\cdot \varphi, \varphi \rangle \rangle)^N.
			\end{equation*} 
			Since $\varphi\in\Gamma( U\Sigma )$ is a solution of \eqref{ec_sitter} it follows that
			\begin{eqnarray}
			\partial_X\langle\langle Y\cdot\varphi, \varphi \rangle\rangle
			&=&\langle\langle Y\cdot\nabla _X \varphi, \varphi \rangle\rangle+\langle\langle Y\cdot \varphi, \nabla _X \varphi \rangle\rangle \nonumber\\
			&=&(id+\tau)\langle\langle Y\cdot\nabla _X \varphi, \varphi \rangle\rangle \nonumber
			\\
			&=&-\frac{1}{2}(id+\tau)\langle\langle Y\cdot\sum_{j=1}^{2}e_j\cdot B(X,e_j)\cdot  \varphi, \varphi \rangle\rangle \nonumber\\
			& & +\frac{1}{2}(id+\tau)\langle\langle Y\cdot X \cdot \nu \cdot  \varphi, \varphi \rangle\rangle \label{nuXnu}.
			\end{eqnarray}
			We first prove the following equality
			\begin{equation}\label{YporB}
			-\frac{1}{2}(id+\tau)\langle\langle Y\cdot\sum_{j=1}^{2}e_j\cdot B(X,e_j)\cdot  \varphi, \varphi \rangle\rangle=\langle\langle B(X,Y)\cdot  \varphi, \varphi \rangle\rangle.
			\end{equation}
		If $X=\sum_{i=1}^{2}x_ie_i$ and $Y=\sum_{k=1}^{2}y_ke_k$, then
			\begin{align*}
			Y\cdot\sum_{j=1}^{2}  e_j\cdot B(X,e_j)&=\sum_{k=1}^{2} \sum_{j=1}^{2}  y_k e_k\cdot e_j\cdot  B(X,e_j)\\
			&=-\sum_{j}   y_j B(X,e_j)+\sum_{k\neq j} y_k e_k\cdot e_j\cdot B(X,e_j)
			\\
			&=-B(X,Y)+\sum_{k\neq j} y_k  e_k\cdot  e_j \cdot B(X,e_j),
			\end{align*}
			therefore,
			\begin{align*}
			\langle\langle Y\cdot\sum_{j=1}^{2} e_j\cdot B(X,e_j)\cdot  \varphi, \varphi \rangle\rangle&=\langle\langle -B(X,Y)\cdot \varphi,\varphi \rangle \rangle\\
			&+\sum_{k\neq j} \langle\langle y_k e_k\cdot  e_j\cdot B(X,e_j) \cdot \varphi, \varphi \rangle \rangle
			\end{align*}
			and
			\begin{align*}
			\tau
			\langle\langle Y\cdot\sum_{j=1}^{2} e_j\cdot B(X,e_j)\cdot  \varphi, \varphi \rangle\rangle&=\langle\langle -B(X,Y)\cdot \varphi,\varphi \rangle \rangle\\
			&-\sum_{k\neq j} \langle\langle y_k e_k\cdot e_j \cdot B(X,e_j)\cdot \varphi, \varphi \rangle \rangle,\end{align*}
			which implies \eqref{YporB}. Using \eqref{nuXnu} and \eqref{YporB}, it is enough to finish the proof to verify that $\frac{1}{2}(id+\tau)\langle\langle Y\cdot X \cdot \nu \cdot  \varphi, \varphi \rangle\rangle$ does not have normal component. We have that
			\begin{eqnarray*}
				\langle\langle Y\cdot X \cdot \nu \cdot  \varphi, \varphi \rangle\rangle&=&\sum_{i\neq k}y_i x_k\tau [\varphi][e_i]\cdot [e_k]\cdot[\nu]\cdot [\varphi]-\sum_{k}y_k x_k\tau [\varphi]\cdot [\nu][\varphi]
			\end{eqnarray*}
			and therefore
			\begin{eqnarray*}	\tau\langle\langle Y\cdot X \cdot \nu \cdot  \varphi, \varphi \rangle\rangle&=&-\sum_{i\neq k}y_i x_k\tau [\varphi][e_i]\cdot [e_k]\cdot[\nu]\cdot [\varphi]-\sum_{k}y_k x_k\tau [\varphi]\cdot [\nu][\varphi],
			\end{eqnarray*}
			which implies
			\begin{eqnarray*}
				\frac{1}{2}(id+\tau)\langle\langle Y\cdot X \cdot \nu \cdot  \varphi, \varphi)\rangle\rangle &=&-\sum_{k}y_k x_k\tau [\varphi]\cdot [\nu]\cdot[\varphi], \end{eqnarray*}
			which has no component normal to the immersion because for $N$ a section of $E$, $[N]$ and $[\nu]=e_4$ are orthogonal. 
		\end{proof}

		We finally prove $2\implies 1.$ Suppose that $M$ is isometrically immersed in $\mathbb{S}_1^3\subset \R^{1,3}$ and consider the section $\varphi\in \Sigma \R^{1,3}\vert_{M}$ given by $\varphi=[1_{Spin(1,3)}, 1_{Cl_{{\R^{^{1,3}}}}}]$; it belongs to $U\Sigma$ and using the spinorial Gauss formula we deduce that it satisfies \eqref{ec_sitter}. 
	\end{proof}

	\subsection{Representation of a surface in $\mathbb{H}_1^3$}
	The pseudo-hyperbolic space of dimension $3$ is the Lorentzian hypersurface of $\R^{2,2}$ with constant sectional curvature $-1$ given by the quadric
	\begin{equation*}
	\mathbb{H}_1^{3}=\{(x_1,x_2,x_3,x_4)\in \R^{2,2} : -x_1^2-x_2^2+x_3^2+x_4^2=-1\}.
	\end{equation*}  
We suppose that $M$ is a simply connected Riemannian surface and $E:=M\times \R$ is the trivial bundle with metric $-d\nu^2$ in the fibers. Let us assume that $B:TM\times TM \longrightarrow E$ is a given symmetric bilinear form.
	
	We consider the splitting $\R^{2,2}=\R e_1 \oplus \R^{1,2}$ where $e_1$ is the first vector of the standard orthonormal basis of $\R^{2,2}$ so that $\langle e_1,e_1\rangle =-1$, the application
	\begin{eqnarray*}
		Spin(0,2)\times Spin(1,0)&\longrightarrow& Spin(1,2)\subset Spin(2,2)\\
		(g_1,g_2)&\longmapsto& g_1 g_2
	\end{eqnarray*}
	where the inclusion $Spin(1,2)\subset Spin(2,2)$ corresponds to the inclusion $\R^{1,2}\subset \R^{2,2}$ in the splitting above, and the representation
	\begin{eqnarray*}
		\rho:\hspace{.5cm}Spin(0,2)\times Spin(1,0)&\longrightarrow& Gl(Cl_{2,2})\\
		(g_1,g_2)&\longmapsto& \begin{array}[t]{rcc}
			\rho(g_1,g_2):Cl_{2,2}&\longrightarrow &Cl_{2,2}\nonumber\\
			v &\longmapsto& g_1 g_2\cdot v.
		\end{array}
	\end{eqnarray*}
	Denoting by $\overset{\sim}{Q}$ a product of spin structures of $TM$ and $E,$ we define the bundles
	\begin{equation*}
	\Sigma:=\overset{\sim}{Q}\times_{\rho}Cl_{2,2},\hspace{.5cm}U\Sigma :=\overset{\sim}{Q}\times_{\rho}Spin{({2,2})}\subset \Sigma \hspace{.3cm}  \text{and}\hspace{.3cm} Cl_{\Sigma}:={\overset{\sim}{Q}\times_{Ad}Cl_{2,2}}
	\end{equation*}
	and set $\nu:=[\overset{\sim}{s},e_1]\in Cl_{\Sigma}$. The proof of the following theorem is similar to the proof of Theorem \ref{inmersion_sitter} and we omit it.
	\begin{thm}\label{inmersion_anti_sitter}
		The following statements are equivalent:
		\begin{enumerate}
			\item There exists $\varphi \in \Gamma(U\Sigma) $ such that
			\begin{equation}\label{ec_anti_sitter}
			\nabla_X\varphi =
			-\frac{1}{2}\sum_{j=1}^{2}e_j\cdot B(X,e_j)\cdot \varphi-\frac{1}{2}X\cdot \nu \cdot \varphi,
			\end{equation}
			for all $X\in TM$.
			\item There exists an isometric immersion $F$ of $M$ into $\mathbb{H}_{1}^3$ with second fundamental form $B$.
		\end{enumerate}
		Moreover $F:M\longrightarrow \mathbb{H}_{1}^3 $ is given by $F=\langle \langle\nu \cdot \varphi, \varphi \rangle\rangle.$
	\end{thm}
	\section{Representation of a surface in $\mathbb{R}_{-}\times \mathbb{S}^2$ }\label{product_nogrupo}
	
	The product manifold $ \R_{-} \times \mathbb{S}^2 $ is isometrically immersed in $\R^{1,3}$ as the following quadric
	\begin{equation*}
	\R_{-} \times \mathbb{S}^2=\{(t,x_1,x_2,x_3)\ :\ x_1^2+x_2^2+x_3^2=1\}\subset \R^{1,3}.
	\end{equation*}

	If $M$ is a surface isometrically immersed in $\R_{-}\times \mathbb{S}^2\subset \R^{1,3}$ then, by the Gauss formula, the restriction to $M$ of a constant spinor field $\varphi$ satisfies
	\begin{equation}\label{gauss_RS2}
	\nabla_X\varphi = -\frac{1}{2}\sum_{j=1}^{2}e_j\cdot B(X,e_j)\cdot \varphi+\frac{1}{2}X_0\cdot \nu \cdot \varphi,
	\end{equation}
	for all $X\in TM$, where $B$ is the second fundamental form of the immersion $M\hookrightarrow\R_{-} \times \mathbb{S}^2$, $\nu$ is the outward normal vector field of $\R_{-} \times \mathbb{S}^2$ in $\R^{1,3}$ and $X_0$ is the projection of $X$ onto $T\mathbb{S}^2$.
	Suppose that $(M,g)$ is a simply connected Riemannian surface and $E=M\times \R$ is the trivial bundle with the metric $-d\nu^2$ in the fibers.
	
We consider the splitting $\R^{1,3}=\R^{1,2}\oplus  \R e_4$, where $e_4$ is the last vector of the standard orthonormal basis of $\R^{1,3}$, the application
	\begin{eqnarray*}
		Spin(0,2)\times Spin(1,0)&\longrightarrow& Spin(1,2)\subset Spin(1,3)\\
		(g_1,g_2)&\longmapsto& g_1 g_2
	\end{eqnarray*}
	and the representation
	$$\begin{array}{ccl}
	\rho:\hspace{.5cm}Spin(0,2)\times Spin(1,0)&\longrightarrow  &Gl(Cl_{1,3})\\
	(g_1,g_2)&\longmapsto& \rho(g_1,g_2):\begin{array}[t]{ccc}
	Cl_{1,3}&\longrightarrow &Cl_{1,3}\\
	v&\longmapsto& g_1g_2v.
	\end{array}
	\end{array} $$
	If $\overset{\sim}{Q}$ stands for a product of spin structures of $TM$ and $E,$ we define the bundles
	\begin{equation*}\label{haz_grande_superficie_RS2}
	\Sigma:=\overset{\sim}{Q}\times_{\rho}Cl_{1,3},\hspace{.5cm}U\Sigma :=\overset{\sim}{Q}\times_{\rho}Spin{({1,3})}\subset \Sigma,\hspace{.5cm}Cl_{\Sigma}:=\overset{\sim}{Q}\times_{Ad}Cl_{1,3}
	\end{equation*}
	and set $\nu:=[\overset{\sim}{s},e_4]\in Cl_{\Sigma}.$

	We consider a symmetric bilinear form $B:TM\times TM\longrightarrow E$ and denote by $S:TM\longrightarrow TM $ the symmetric operator such that $\langle S(X),Y\rangle = \langle B(X,Y),N\rangle$ for all $X,Y\in TM$ where $N$ is a given unit section of $E$. Let us suppose that there exist $T\in \Gamma(TM)$ and $f\in C^{\infty}(M)$ such that
	\begin{eqnarray}\label{compatibilidad_prod}
	\vert \vert T\vert \vert^2 - f^2=-1\label{norma_vect_killing}\\
	\nabla_{X}T=f S(X)\label{compatibilidad_prod_2}
	\\
	df(X)=\langle S(X),T\rangle. \label{compatibilidad_prod_3}
	\end{eqnarray} 
	
	It follows from \eqref{norma_vect_killing} that $e_0:=T+fN$ has norm $-1$.
	
	\begin{lem}\label{eta}
		If $M$ is simply connected and $T\in \Gamma(TM)$ satisfies the equations \eqref{norma_vect_killing}--\eqref{compatibilidad_prod_3} then there exists $\eta: M\longrightarrow \R$ such that $d\eta(X)=-\langle X,T\rangle$.
	\end{lem}
	\begin{proof}
		The latter holds since, for all $X,Y\in \Gamma(TM)$,
		\begin{eqnarray*}
			d\beta(X,Y)&=& X.\beta(Y)- Y.\beta(X)-\beta([X,Y])\\
			&=&\langle  \nabla_X Y,T\rangle +\langle  Y,\nabla_X T\rangle-\langle  \nabla_Y X,T\rangle -\langle X,\nabla_ YT\rangle- \langle [X,Y], T \rangle\\
			&=&\langle  Y, fS(X)\rangle-\langle X,fS(Y)\rangle\\
			&=&0.
		\end{eqnarray*}
	\end{proof}
	\begin{thm}\label{thm_inmersion_rs2} If $M$ is simply connected and $\eta:M\longrightarrow \R$ is the differentiable function of Lemma \ref{eta} then the following statements are equivalent:
		\begin{enumerate}
			\item There exists $\varphi \in \Gamma(U\Sigma) $ such that
			\begin{equation}\label{gauss_producto_RS2}
			\nabla_X \varphi=-\frac{1}{2}\sum_{j=1}^2e_j\cdot B(X,e_j)\cdot \varphi +\frac{1}{2}X_0\cdot \nu \cdot \varphi,
			\end{equation}
			for all $X\in TM$ and $X_0=X+\langle X,e_0\rangle e_0$.
			\item There exists an isometric immersion $F$ of $M$ into $\R_{-} \times \mathbb{S}^2$ with second fundamental form $B$.
		\end{enumerate}
		Moreover,  $F:M\longrightarrow\R_{-} \times \mathbb{S}^2$ is given by 
		\begin{equation}\label{prod_nogrupo}
		F=\eta\langle \langle e_0 \cdot \varphi,\varphi\rangle \rangle + \langle \langle \nu\cdot \varphi,\varphi\rangle \rangle.
		\end{equation}
	\end{thm}
For the proof we will need the following lemma:		
\begin{lem}\label{e0_constante}
If $\varphi\in \Gamma(U\Sigma)$ is a solution of (\ref{gauss_producto_RS2}), then, for all $X\in TM,$
 $$\partial_ X \langle \langle e_0\cdot \varphi,\varphi\rangle \rangle=0.$$
\end{lem}
\begin{proof}
We have by definition that
		\begin{equation*}\label{interna}
	\partial_ X \langle \langle e_0\cdot \varphi,\varphi\rangle \rangle= \langle\langle \nabla_X e_0\cdot \varphi, \varphi \rangle \rangle +  (id+\tau) \langle\langle  e_0\cdot \nabla_X\varphi, \varphi \rangle \rangle.
	\end{equation*}
	Following 
	the proof of Proposition \ref{isometria_haznormal_sitter} we find that
	\begin{eqnarray}
	(id + \tau)\langle\langle  e_0\cdot \nabla_X\varphi, \varphi \rangle \rangle
	&=& \langle\langle  B(X,T)\cdot \varphi, \varphi \rangle \rangle+  (id + \tau) \langle\langle  fN\cdot \nabla_X\varphi, \varphi \rangle \rangle. \label{cero_prod}
	\end{eqnarray}
	Since $\varphi$ satisfies equation \eqref{gauss_producto_RS2} we have that
	\begin{align*}
	(id + \tau) \langle\langle  fN\cdot \nabla_X\varphi, \varphi \rangle \rangle&=\frac{1}{2}(id + \tau) \langle\langle  fN\cdot (-\sum_{j=1}^2e_j\cdot B(X,e_j)\cdot\varphi+ X_0\cdot \nu \cdot \varphi), \varphi \rangle \rangle.
	\end{align*}
We simplify the above expression replacing
	\begin{eqnarray*}
		\langle \langle-\frac{f}{2} N\cdot (\sum_{j=1}^2e_j\cdot B(X,e_j)\cdot\varphi ), \varphi \rangle \rangle&=&-\frac{f}{2}\langle \langle S(X)\cdot \varphi, \varphi \rangle \rangle
	\end{eqnarray*}
	which implies
	\begin{equation}\label{uno_prod}
	\frac{1}{2}(id + \tau) \langle\langle  fN\cdot (-\sum_{j=1}^2e_j\cdot B(X,e_j)\cdot\varphi ), \varphi \rangle \rangle=-f\langle \langle S(X)\cdot \varphi, \varphi \rangle \rangle.
	\end{equation}
	On the other hand,
	$$\tau\langle\langle fN\cdot X_0\cdot \nu \cdot \varphi, \varphi \rangle \rangle=-\langle\langle fN\cdot X_0\cdot \nu \cdot \varphi, \varphi \rangle \rangle$$ and consequently,
	\begin{equation}\label{dos_prod}
	\frac{1}{2}(id + \tau) \langle\langle fN\cdot X_0\cdot \nu \cdot \varphi, \varphi \rangle \rangle=0.
	\end{equation}
	From the equations \eqref{cero_prod}, \eqref{uno_prod} and \eqref{dos_prod} it follows that
	\begin{eqnarray*}
		(id + \tau) \langle\langle  e_0\cdot \nabla_X\varphi, \varphi \rangle \rangle&=&  \langle\langle  B(X,T)\cdot \varphi, \varphi \rangle \rangle -f\langle \langle S(X)\cdot \varphi, \varphi \rangle \rangle.
	\end{eqnarray*}
	Finally, the conditions \eqref{compatibilidad_prod_2} and  \eqref{compatibilidad_prod_3} imply that
	\begin{equation*}
	\nabla_X e_0=f S(X)+\langle S(X), T \rangle N=f S(X)-B(X,T),
	\end{equation*} 
	and therefore
	$$ \langle\langle \nabla_X e_0\cdot \varphi, \varphi \rangle \rangle +  (id+\tau) \langle\langle  e_0\cdot \nabla_X\varphi, \varphi \rangle \rangle=0,$$
	which ends the proof of the lemma.
	\end{proof}

	The proof of the theorem then relies on the following lemma whose proof will be omitted since it is very similar to the proofs of Lemma \ref{lemadFX_Sitter} and Proposition \ref{isometria_haznormal_sitter} (see \cite{TBZJ} for more details).
	\begin{lem}\label{lema_inmersion_rs2}
		The function $F$ defined by \eqref{prod_nogrupo} satisfies $dF(X)=\langle\langle  X\cdot \varphi, \varphi \rangle \rangle$ for all $X\in TM$.
	\end{lem}
	\begin{proof}[Sketch of the proof of Theorem \ref{thm_inmersion_rs2}.]
		Suppose that there exists a solution $\varphi \in \Gamma(U\Sigma )$ of \eqref{gauss_producto_RS2}, then $F$ defined by (\ref{prod_nogrupo}) takes values in $\R_{-}\times \mathbb{S}^2:$ it follows from Lemma \ref{e0_constante}
		 that $\langle \langle e_0\cdot \varphi,\varphi\rangle \rangle$ is constant and can be identified after a rigid motion with the first element $e_1$ of the standard basis of $\R^{1,3};$ by definition $ \langle \langle \nu\cdot \varphi,\varphi\rangle \rangle=Ad([\varphi]^{-1})([\nu])$ with $[\varphi]^{-1}\in Spin(1,3),$ so that $ \langle \langle \nu\cdot \varphi,\varphi\rangle \rangle$ has norm $1$; furthermore $\langle \langle \nu\cdot \varphi,\varphi \rangle \rangle $ and $\langle \langle e_0\cdot \varphi,\varphi\rangle \rangle=e_1$ are orthogonal, so $\langle \langle \nu\cdot \varphi,\varphi\rangle \rangle\in \mathbb{S}^2$ and $F(M)\subset \R_{-}\times  \mathbb{S}^2$. By Lemma \ref{lema_inmersion_rs2} we have that $dF(X)= \langle \langle X\cdot \varphi,\varphi \rangle \rangle$ and we can follow the proof of Proposition \ref{isometria_haznormal_sitter} to obtain that $F:M\longrightarrow \R_{-}\times \mathbb{S}^2$ is an isometry and 
		\begin{eqnarray*}
			\Phi_E:\hspace{.5cm} E&\longrightarrow& T(\R_{-}\times \mathbb{S}^{2})\\
			X\in E_m&\longmapsto& (F(m), \langle \langle X\cdot \varphi, \varphi \rangle \rangle)
		\end{eqnarray*}
		is a bundle isomorphism between $E$ and the normal bundle of the immersion $F$ which identifies $B$ to the second fundamental form. The  implication $2\implies 1$ is a consequence of the Gauss formula, equation \eqref{gauss_RS2}.
	\end{proof}
	\section{Representation of surfaces in $\mathbb{L} (\kappa,\tau)$ spaces}\label{inmersiones_lkt}
	In analogy with the Riemannian homogeneous spaces  $\mathbb{E}(\kappa,\tau)$ (see \cite{BD} and \cite{MZNO}), the $3$-dimensional Lorentzian homogeneous spaces $\mathbb{L}(\kappa, \tau)$ are Lorentzian fibrations $\mathbb{L}(\kappa, \tau)\rightarrow M^2(\kappa)$ where $M^2(\kappa)$ is a Riemannian surface with constant sectional curvature $\kappa$; the fibers of the projection are integral curves of a (complete) time-like unit Killing vector field $\xi$ over the total space and $\tau$ is the bundle curvature which is defined as the real number such that
	\begin{equation}
	\overline{\nabla}_X \xi=-\tau X\times \xi 
	\end{equation}
	where $\overline{\nabla}$ is the Levi-Civita connection and $\times$ is the natural cross product in $\mathbb{L}(\kappa,\tau).$ A precise description of these spaces is the following: for $\kappa\in\R,$ defining
	\begin{eqnarray*}
		V:=\left\{(x,y,z)\in \R^3:1+\frac{\kappa}{4}(x^2+y^2)>0\right\}& \ \  \text{and} \ \ &\lambda:=\frac{1}{1+\frac{\kappa}{4}(x^2+y^2)},
	\end{eqnarray*}
	the space is 
	\begin{equation*}
	\mathbb{L}(\kappa,\tau)=\left(V,\  \lambda^2(dx ^2+dy^2)-(\tau\lambda(ydx-xdy)+dz)^2\right).
	\end{equation*}
	Using this representation for $\mathbb{L}(\kappa,\tau)$ and setting $\sigma=\frac{\kappa}{2\tau}$ the frame
	\begin{eqnarray*}
		E_1&=&\lambda^{-1}(cos(\sigma z)\partial_x+sen(\sigma z)\partial_y)+\tau(xsen(\sigma z)-ycos(
		\sigma z))\partial_z\\
		E_2&=& \lambda^{-1}(-sen(\sigma z)\partial_x+cos(\sigma z)\partial_y)+\tau(xcos(\sigma z)-ysen(
		\sigma z))\partial_z\\
		E_3&=&\partial_z
	\end{eqnarray*}
	is orthonormal and such that $\langle E_1,E_1\rangle=\langle E_2,E_2\rangle=1=-\langle E_3,E_3\rangle.$ Moreover, $E_1\times E_2=-E_3,\ E_2\times E_3=E_1,\ E_1\times E_3=-E_2,$ and, for $\Gamma_{ij}^k:=\langle \overline{\nabla}_{_{E_i}}E_j,E_k\rangle$ 
	\begin{equation*}
	\Gamma_{21}^3=\Gamma_{13}^2=\tau=-\Gamma_{12}^3=-\Gamma_{23}^1,\hspace{.5cm}\Gamma_{31}^2=\sigma+\tau=-\Gamma_{32}^1
	\end{equation*}
	and
	\begin{eqnarray}
	[E_1,E_2]=2\tau E_3,&[E_2,E_3]=\sigma E_1,&[E_3,E_1]=\sigma E_2.
	\end{eqnarray}

	We describe in the following table the spaces $\mathbb{L}(\kappa,\tau)$ in terms of the values $\kappa$ and $\tau.$ 
	\begin{center}
		\begin{tabular}{ | c | c | c | c | }
			\hline
			& $\kappa<0$ &$ \kappa=0$ &$\kappa>0$\\  \hline
			$\tau=0$ & $\mathbb{H}^2(\kappa)\times \R_{-}$ &$ \mathbb{L}^3 $& $\mathbb{S}^2(\kappa)\times \R_{-}$\\ \hline	
			\vspace{0.1cm}			
			$\tau\neq 0$ & $\widetilde{SL_2^1}$ & ${Nil_3^1}$& $\mathbb{S}^{3,1}_{Berger}$ \\ \hline
		\end{tabular}
	\end{center}
	Some particular cases occur when $\kappa+4\tau^2=0:$ if $\kappa=\tau=0$ then $\mathbb{L}(\kappa,\tau)=\mathbb{L}^3,$ and in the other case $\mathbb{L}(\kappa,\tau)=\mathbb{H}^3_1(\kappa)$ (in the table above this space corresponds to $\widetilde{SL^1_2}$). Let us also note that the space $\mathbb{S}_1^3$  does not admit any unit Killing vector field and therefore cannot be a $\mathbb{L}(\kappa,\tau)$ space.
	
	We now assume $\tau\neq 0$ and characterize the immersion of a Riemannian surface in $\mathbb{L}(\kappa,\tau)$ using its Lie group structure (see [1] for a similar result in the space $\mathbb{E}(\kappa,\tau)$). The Lie algebra of $\mathbb{L}(\kappa,\tau)$ is $\g=\R^3$ with the Lie bracket given in the canonical basis by
	\begin{eqnarray}
	[e_1^o,e_2^o]=2\tau e_3^o,&[e_2^o,e_3^o]=\sigma e_1^o,&[e_3^o,e_1^o]=\sigma e_2^o.
	\end{eqnarray}
	The metric in $\mathbb{L}(\kappa,\tau)$ is the left-invariant metric $\langle.,.\rangle$ such that $(e_1^o,e_2^o,e_3^o)$ is an orthonormal basis of $\mathfrak{g}$ satisfying $\langle e_1^o,e_1^o\rangle=\langle e_2^o,e_2^o\rangle=1=-\langle e_3^o,e_3^o\rangle$.
	
	Let $M$ be an orientable Riemannian surface and
	$S:TM\rightarrow TM$ a symmetric operator. Let us suppose that there exist $T\in \Gamma(TM)$ and $\nu \in C^{\infty}(M)$ such that for all $X\in TM$
	\begin{eqnarray}
	&\vert T \vert^2-\nu^2=-1,\label{comp_0}\\
	&\nabla_{_{X}}T=\nu(S(X)+\tau J(X)) \label{comp_1},\\
	&d\nu(X)=\langle S(X)+\tau J(X),T\rangle,\label{comp_2}
	\end{eqnarray}
	where $J:TM\rightarrow TM$ is the rotation of angle $+{\pi}/{2}$ in the tangent planes.
	The following theorem is similar to Theorem 5 in [1] (and we omit the proof here).
	\begin{thm}\label{inmersion_lkt}
		If $M$ is simply connected, the following statements are equivalent:
		\begin{enumerate}
			\item There exists $\psi \in \Gamma(\Sigma M)$ solution of
			\begin{equation}\label{intrinseca_lkt}
			\nabla_{_X}\psi=\frac{i}{2}S(X)\cdot \psi -\frac{1}{2}(i\tau X+\langle X,T\rangle(\sigma +2\tau)(iT+\nu))\cdot \omega \cdot \psi,
			\end{equation}
			for all $X\in TM$ and such that $\vert \psi^+\vert^2-\vert \psi^-\vert^2=1$. Here $\omega=\epsilon_1\cdot \epsilon_2$ where $(\epsilon_1,\epsilon_2)$ is a positively oriented orthonormal basis of $M$.
			\item There exists an isometric immersion of $M$ into $\mathbb{L}(\kappa,\tau)$ with shape operator $S$.
		\end{enumerate}
	\end{thm}
	As a consequence of that theorem let us deduce a representation theorem for an immersion in the 3-dimensional anti de Sitter space, when we consider this space as the group of $2\times 2$ matrices
	\begin{equation*}
	SU_2^1=\left\{\left(\begin{array}{cc}
	z&\omega\\
	\overline{\omega}& \overline{z}
	\end{array}\right)\in M_2(\mathbb{C})\ :\ \vert z\vert^2 - \vert\omega \vert^2=-1\right\},
	\end{equation*}
	whose Lie algebra is
	\begin{equation*}
	\left\{\left(\begin{array}{cc}
	i\lambda & a\\
	\overline{a}& -i\lambda
	\end{array}\right)\ : a\in \mathbb{C}, \ \lambda\in \R \right\}.
	\end{equation*}
	A basis for this Lie algebra is
	\begin{eqnarray*}
		E_1=\left(\begin{array}{cc}
			0 & 1\\
			1& 0
		\end{array}\right), &E_2=\left(\begin{array}{cc}
			0 & -i\\
			i& 0
		\end{array}\right),&E_3=\left(\begin{array}{cc}
			i & 0\\
			0&-i
		\end{array}\right),
	\end{eqnarray*}
	the Lie bracket is determined by
	\begin{eqnarray*}
		[E_1,E_2]=2E_3,&[E_2,E_3]=-2E_1,&[E_3,E_1]=-2E_2,
	\end{eqnarray*}
	and the Lorentzian metric of $SU_1^2$ is the left-invariant metric such that $\langle E_i,E_j\rangle=0$ if $i\neq j$ and $\langle E_1,E_1\rangle=\langle E_2,E_2\rangle=1=-\langle E_3,E_3\rangle$. The non-trivial solution $\kappa=-4$ and $\tau=1$ of $\kappa+4\tau^2=0$ gives that $\mathbb{L}(\kappa,\tau)=SU_2^1$ and from Theorem \ref{inmersion_lkt} we deduce the following representation theorem in this space:
	\begin{thm}\label{anti_de_sitter_como_grupo}Let $M$ be a simply connected Riemannian surface and   $S:TM\rightarrow TM$ a symmetric operator.
		Let us suppose that there exist $T\in\Gamma(TM)$ and $\nu\in C^{\infty}(M)$ solutions of \eqref{comp_0}-\eqref{comp_2} with $\tau=1.$ The following statements are equivalent:
		\begin{enumerate}
			\item  There exists $\psi \in \Gamma(\Sigma M)$ solution of 
			\begin{equation}\label{intrinseca_lkt_anti_sitter}
			\nabla_{_X}\psi=\frac{i}{2}S(X)\cdot \psi -\frac{1}{2}i X\cdot \omega \cdot \psi,
			\end{equation}
			for all $X\in TM$ and such that $\vert \psi^+\vert^2-\vert \psi^-\vert^2=1$.
			\item There exists an isometric immersion of $M$ into $SU_2^1$ with shape operator $S$.
		\end{enumerate}
	\end{thm}
	\section{Correspondences between CMC surfaces}\label{sec corr}
	\subsection{Correspondence between CMC surfaces in $\R^{1,2}$ and in $\mathbb{H}_1^{3}$}\label{sec_r12_h13}
	Let $(M,g)$ be an oriented Riemannian surface and denote by $\langle \ , \ \rangle$ the real part of the natural Hermitian product in $\Sigma M$. Recall that for $X \in TM$ and $\psi_1,\psi_2$ sections of $ \Sigma M$ we have 
	\begin{eqnarray*}
		\langle X\cdot \psi_1 , X\cdot \psi_2 \rangle= \langle \psi_1 ,\psi_2 \rangle\ \text{if}\  \vert X\vert =1 &\ \text{and}\ &\langle X\cdot \psi_1 , \psi_2 \rangle=- \langle \psi_1 ,X\cdot \psi_2 \rangle .
	\end{eqnarray*}
	Note that if $\psi=\psi^++\psi^-$ in the decomposition $\Sigma M=\Sigma^+ M\oplus\Sigma^- M$ then
	$\langle \psi , \psi \rangle =\vert \psi^+\vert^2+\vert \psi^- \vert^2$ and if $\overline{\psi}=\psi^+ - \psi^-$ then $\psi$ and $\overline{\psi}$ are related by $ie_1\cdot e_2\cdot \psi= \overline{\psi}$ for $(e_1,e_2)$ a positively oriented orthonormal frame of $TM$ and therefore,
	\begin{equation}\label{psibarrapsi}
	\langle \psi, ie_1\cdot e_2\cdot \psi\rangle=\langle \psi,\overline{\psi}\rangle= \vert \psi^+\vert^2-\vert \psi^- \vert^2.
	\end{equation}
	
	Here, we first prove that the Killing-type equation characterizing the immersions of $M$ in the $\mathbb{L}(\kappa,\tau)$ spaces with  $\kappa+4\tau^2=0$ (see Theorem \ref{inmersion_lkt}) is equivalent to its associated Dirac equation and deduce that the immersions in these spaces are also characterized by Dirac equations. Let $H:M\longrightarrow \R$ be a differentiable function and $\tau$ a real number.
	
	\begin{thm}\label{killing_dirac_lkt}
		There exists a correspondence between the following data:
		\begin{enumerate}
			\item A section $\psi$ of $\Sigma M$ with $\vert \psi^+\vert^2-\vert \psi^- \vert^2=1$ solution of the Dirac equation
			\begin{equation}\label{dirac_lkt}
			D\psi= - iH\psi + i \tau \omega \cdot \psi.
			\end{equation}
			\item A pair $(\psi, S)$, where $S$ is a symmetric operator such that $H=\frac{1}{2} trS$ and $\psi\in \Gamma(\Sigma M)$ is a solution of
			\begin{equation}\label{killing_lkt}
			\nabla_X\psi=\frac{i}{2}S(X)\cdot \psi - \frac{i}{2}\tau X \cdot \omega \cdot \psi\ \ 
			\end{equation}
			for all $X\in TM$ such that $\vert \psi^+\vert^2-\vert \psi^- \vert^2=1$.
		\end{enumerate}
	\end{thm}	
	To prove this theorem we use the following lemma whose proof can be found in \cite{TBZJ} (see Lemma 9.1.2).
	\begin{lem}\label{operador_forma_lkt}
		If $\psi\in\Gamma(\Sigma M)$ with $\vert \psi^+\vert^2-\vert \psi^- \vert^2=1$ satisfies (\ref{dirac_lkt}) then the operator $S:TM\longrightarrow TM$ given by
		\begin{equation}
		\langle S(Y), X\rangle= \frac{2}{\vert \psi \vert^2} \left(\langle i X\cdot \nabla_Y \psi , \psi \rangle -\frac{\tau}{2}g(X,JY) \vert \psi \vert^2 \right)\ \ \ \ X, Y\in TM ,
		\end{equation}
		is symmetric and its trace is $H=\frac{1}{2}trS$; here $J$ still denotes the rotation of angle $+\pi/2$ in the tangent planes of $M.$
	\end{lem}
	\begin{proof}[Proof of Theorem \ref{killing_dirac_lkt}.]The proof of $(2)\Rightarrow (1)$ readily follows from the definition of the Dirac operator and we omit it. If $\psi\in \Gamma(\Sigma M)$ is a solution of equation \eqref{dirac_lkt} with $\vert \psi^+ \vert^2-\vert \psi^- \vert^2=1$ then $(\frac{i}{\vert \psi \vert}\psi,\frac{i}{\vert  \psi \vert} e_1\cdot \psi, \frac{i}{\vert  \psi \vert} e_2\cdot \psi, \frac{i}{\vert  \psi \vert}  e_1\cdot e_2 \cdot \psi )$ is an orthonormal basis of $\Sigma M$ with respect to $\langle\ ,\ \rangle$ and therefore we have
		\begin{align}\label{nablae1psi}
		\nabla_{e_1}\psi&=\frac{1}{\vert \psi \vert^2}	\langle \nabla_{e_1}\psi, i\psi\rangle i\psi + \frac{1}{\vert \psi \vert^2}\langle \nabla_{e_1}\psi, i e_1\cdot \psi \rangle i e_1\cdot \psi \\
		&+ \frac{1}{\vert \psi \vert^2} \langle \nabla_{e_1}\psi, i e_2\cdot \psi \rangle ie_2\cdot \psi + \frac{1}{\vert \psi \vert^2} \langle \nabla_{e_1}\psi,  i  e_1\cdot e_2 \cdot \psi \rangle i e_1\cdot e_2 \cdot \psi.\nonumber
		\end{align}
		We first note that, for $k=1,2,$ $\langle\nabla_{e_k}\psi,ie_1\cdot e_2\cdot\psi\rangle=0.$ This is a direct consequence of \eqref{psibarrapsi} with $|\psi^+|^2-|\psi^-|^2=1$ and elementary properties of $\langle.,.\rangle.$ We then deduce that $\langle\nabla_{e_1}\psi,i\psi\rangle=0$ using the Dirac equation \eqref{dirac_lkt}:
		\begin{eqnarray*}
			\langle \nabla_{e_1}\psi, i\psi\rangle&=&\langle iH\psi + - i\tau e_1\cdot \omega \cdot \psi +e_1\cdot
			\nabla_{e_1}\psi, ie_1 \cdot \psi\rangle\\
			&=&-\langle D\psi - e_1\cdot \nabla_{e_1}\psi, i e_1 \cdot \psi\rangle\\
			&=&-\langle e_2\cdot \nabla_{e_2}\psi, i e_1 \cdot \psi\rangle 
			\\
			&=&-\langle \nabla_{e_2}\psi, i  e_1\cdot e_2 \cdot \psi\rangle\\
			&=&0.
		\end{eqnarray*}
		So, the expression \eqref{nablae1psi} of $\nabla_{e_1}\psi$ reduces to \begin{equation*}
		\nabla_{e_1}\psi= \frac{1}{\vert \psi \vert^2}(\langle \nabla_{e_1}\psi, i e_1\cdot \psi \rangle i e_1\cdot \psi + \langle \nabla_{e_1}\psi, i e_2\cdot \psi \rangle ie_2\cdot \psi)
		\end{equation*} and from Lemma \ref{operador_forma_lkt} we have that
		\begin{eqnarray*}
			\nabla_{e_1}\psi&=& \frac{i}{\vert \psi \vert^2}\left(\langle i e_1\cdot   \nabla_{e_1}\psi,  \psi \rangle  e_1  + \langle i  e_2\cdot  \nabla_{e_1}\psi, \psi \rangle e_2\right)\cdot \psi \\
			&=&   \frac{i}{2}\left( \langle S(e_1), e_1\rangle e_1+( \langle S(e_1), e_2\rangle+\tau g(e_2,J(e_1))) e_2\right)  \cdot \psi\\
			&=& \frac{i}{2}S(e_1)\cdot \psi - \frac{i}{2}\tau e_1\cdot e_1 \cdot e_2 \cdot \psi\\
			&=& \frac{i}{2}S(e_1)\cdot \psi - \frac{i}{2}\tau e_1\cdot\omega \cdot \psi.
		\end{eqnarray*}
		We prove in a similar way that $\nabla_{e_2}\psi=\frac{i}{2}S(e_2)\cdot \psi - \frac{i}{2}\tau e_2\cdot\omega \cdot \psi.$ 
	\end{proof}
	As a consequence of Theorems \ref{inmersion_lkt} and \ref{killing_dirac_lkt} we have the following proposition.
	\begin{prop}\label{inmersion_lkt_dirac}
		If $(M,g)$ is a simply connected Riemmanian surface then the following statements are equivalent:
		\begin{enumerate}
			\item There exists $\psi \in \Gamma(\Sigma M)$ with $\vert \psi ^+\vert^2-\vert \psi ^-\vert^2=1$ solution of the equation
			$$D\psi= -iH\psi + i\tau \omega \cdot \psi.$$
			\item There exists an isometric immersion of mean curvature $H$ of $M$ into  $\mathbb{L}(\kappa, \tau )$, where $\kappa +4\tau^2=0$.
		\end{enumerate}
	\end{prop}
	
	Now we prove that CMC surfaces in $\R^{1,2}$ naturally correspond to CMC surfaces in $\mathbb{H}^3_1.$ Since $\R^{1,2}$ and  $\mathbb{H}^3_1$ are $\mathbb{L}(\kappa,\tau)$ spaces, this correspondence is similar to the Lawson type correspondences in the $\mathbb{E}(\kappa,\tau)$ spaces found by Daniel in \cite{BD}.
	
	\begin{prop}\label{correspondencia_minkowski_sitter}
		Given an orientable Riemannian surface $(M,g)$ and for each $H_1\in (-\infty,-1]\cup [1,\infty)$ there exists a correspondence between the spinor fields $\psi_1$ of $\Sigma M$ with $\vert \psi_1^+\vert ^ 2 - \vert \psi_1^-\vert ^ 2=1$ which satisfy
		\begin{equation}\label{prop dirac corresp}
		D\psi_1=-iH_1\psi_1
		\end{equation}
		and the spinor fields $\psi_2$ with $\vert \psi_2^+\vert ^ 2 - \vert \psi_2^-\vert ^ 2=1$ which satisfy
		\begin{equation}\label{terminar}
		D\psi_2=-iH_2\psi_2+i\omega \cdot \psi_2
		\end{equation}
		for some $H_2 \in\R$. 
	\end{prop}
	\begin{proof}
		Let $\psi_1\in\Gamma(\Sigma M)$ be a solution of (\ref{prop dirac corresp}) such that $\vert \psi_1^+\vert ^ 2 - \vert \psi_1^-\vert ^ 2=1.$ If $a=cos\theta + sen \theta e_1\cdot e_2$ and $\psi_2:=a\cdot \psi_1$ then
		\begin{eqnarray*}
			D\psi_2&=&-iH_1cos2\theta\cdot\psi_2 +i H_1sen 2\theta e_1\cdot e_2\cdot\psi_2.
		\end{eqnarray*}
		Taking $\theta\in\R$ such that $\sin 2\theta=1/H_1,$ the last equation is equivalent to
		\begin{equation*}
		D\psi_2=-i (\pm \sqrt{H_1^2-1})\psi_2+i\omega \cdot \psi_2
		\end{equation*}
		which is of the form \eqref{terminar}.
	\end{proof}

	As a consequence of Propositions \ref{inmersion_lkt_dirac} and \ref{correspondencia_minkowski_sitter} we have the following result.
	
	\begin{cor}\label{coro_correspondabnce_r12_h13}
		For $H_1\in (-\infty,-1]\cup[1,\infty)$ there exists a correspondence between immersions of mean curvature $H_1$ of $M$ in $\R^{1,2}$ and immersions of mean curvature $\pm \sqrt{H_1^2-1}$ of $M$ in $\mathbb{H}_1^3$.
	\end{cor}
	
	\subsection{Correspondence between minimal surfaces in $\R^3$ and maximal surfaces in $\R^{1,2}$}\label{sec_calabi}
	By \cite{TFI}, a minimal immersion of a Riemannian surface in $\R^3$ is characterized by the existence of a spinor field $\psi_1$ such that
	\begin{eqnarray}\label{diracr3}
	\vert \psi_1^+\vert^2+\vert \psi_1^-\vert^2=1&\text{and}&D\psi_1=0.
	\end{eqnarray}
	
	Meanwhile, Proposition \ref{inmersion_lkt_dirac} implies that a maximal immersion of a Riemannian surface in $\R^{1,2}$ is given by a spinor field $\psi_2$ such that 
	\begin{eqnarray}\label{diracr12}
	\vert \psi_2^+\vert^2-\vert \psi_2^-\vert^2=1&\text{and}&D\psi_2=0.
	\end{eqnarray}
	\begin{prop}\label{calabi_espinores}
		Let $(M,g)$ be an orientable Riemannian surface. A solution $\psi_1\in\Gamma(\Sigma M)$ of \eqref{diracr3} such that
		\begin{equation}\label{psi1 positive}
		\vert \psi_1^+\vert^2-\vert \psi_1^-\vert^2>0
		\end{equation}
		naturally corresponds to a solution $\psi_2\in\Gamma(\Sigma \overline{M})$ of \eqref{diracr12}, where $\Sigma \overline{M}$ is a spinor bundle conformal to $\Sigma M.$
	\end{prop}
	\begin{rem}
	 It is well known that the components of $\psi_1$ in \eqref{diracr3} depend on the Gauss map of $M$ in $\R^3;$ by explicit calculations, it is easy to show that (\ref{psi1 positive}) means that the Gauss map of $M$ takes values in the lower hemisphere of $\mathbb{S}^2.$
	\end{rem}
	\begin{proof}
		Let $\psi_1\in\Gamma(\Sigma M)$ be a solution of \eqref{diracr3}. We define $\psi_2$ in terms of $\psi_1$ as
		$$\psi_2:=\frac{1}{\small{\sqrt{\vert \psi_1^+\vert^2-\vert \psi_1^-\vert^2}}}\ \widetilde{\psi_1},$$
		where $\widetilde{\psi_1}$ is the spinor field $\psi_1$ in the spinor bundle of $M$ associated to the conformal metric $(\vert \psi_1^+\vert^2-\vert \psi_1^-\vert^2)^2g$. Then $\vert \psi_2^+\vert^2 -\vert \psi_2^-\vert^2=1,$ and the conformal invariance of the Dirac operator (see \cite[Pag. 69]{BHMM}) implies (\ref{diracr12}). 
	\end{proof}

	We directly deduce the following corollary:
	
	\begin{cor}\label{coro_correspondencia_calabi}
		There exists a conformal correspondence between minimal surfaces in $\R^3$ with Gauss map image in an hemisphere and maximal surfaces in $\R^{1,2}$.
	\end{cor}
	\begin{rem}
		It can be shown by explicit calculations that the correspondence $\psi_1\mapsto\psi_2$ corresponds to the transformation
		$$\int(\Phi_1,\Phi_2,\Phi_3)dz\mapsto \int(i\Phi_1,i\Phi_2,\Phi_3)dz$$
		on Weierstrass data of immersions in $\R^3$ and in $\R^{1,2}$ (details are carried out in \cite[Sec 9.2]{TBZJ}). This is the correspondence described in \cite{LLS}. We point out that for the general correspondence no hypotheses on the Gauss map image of the minimal surface is needed, see  \cite{LLS}.
	\end{rem}

	\appendix
	\section{Identification of spinor bundles}\label{spinor_bundle_identification}
	We consider a simply connected pseudo-Riemmanian surface $M$ and the trivial bundle $E=M\times \R \longrightarrow M$ with metric $\mp d\nu ^2$ depending on whether $M$ is Riemannian or Lorentzian and let us denote the spin structures of $TM$ and $E$ by $\widetilde{Q}_M$ and $\widetilde{Q}_E$ respectively. We identify here the spinor bundle $\Sigma M$ with a subbundle of the bundle $\Sigma=\left(\widetilde{Q}_{_M} \times_M \widetilde{Q}_{_E}\right)\times_\rho Cl_{1,2}$ defined by \eqref{haz_prod_tensorial_riem}.
	
	Let us begin with the Riemannian case. We define the complexified quaternions as
	\begin{equation*}
	\mathbb{H}^{\mathbb{C}}:=\mathbb{H}\otimes_{\R}\C=\{z_0+z_1I+z_2J+z_3K : z_i\in \mathbb{C}\}.
	\end{equation*}
	This algebra is endowed with the quadratic form $\langle q,q\rangle=z_0^2+z_1^2+z_2^2+z_3^2$, where $q=z_0+z_1I+z_2J+z_3K\in\mathbb{H}^{\mathbb{C}}$.
	
	\begin{rem}\label{rmk app cl12}
		Let $(e_0,e_1,e_2)$ be the standard basis of $\R^{1,2}$, which is such that $\langle e_i, e_j\rangle=0$ if $i\neq j$ and $\langle e_0, e_0\rangle=-1=-\langle e_1, e_1\rangle=-\langle e_2, e_2\rangle$. The map obtained by linearity from $e_0\longmapsto iI$,
		$e_1\longmapsto J$,
		$	e_2\longmapsto JI=-K$ is a Clifford application which induces an isomorphism of algebras between $Cl_{1,2}$ and $\mathbb{H}^{\mathbb{C}}$. Using this isomorphism, the even Clifford algebra $Cl_{1,2}^{0}$ is isomorphic to
		\begin{equation*}
		\{q_0+q_1I+iq_2J+iq_3K :  q_i\in \R \}.
		\end{equation*}
	\end{rem}
	
	We consider the following representations of the group $Spin(0,2)$ given by the left-multiplication:
	\begin{eqnarray*}
		\rho_1:Spin(0,2)\longrightarrow Gl(Cl_{0,2})& \text{and}& \rho_2:Spin(0,2)\subset Spin(1,2)\longrightarrow Gl(Cl_{1,2}^{0}).
	\end{eqnarray*}
	Here and below we use the models 
	$$Cl_{0,2}=\mathbb{H}\hspace{.5cm}\mbox{and}\hspace{.5cm}Spin(0,2)=\{q_0+q_1I: q_i\in\R,\ q_0^2+q_1^2=1\}.$$ 
	The representation $\rho_1$ is the standard spin representation (if we identify $\mathbb{H}\simeq \C^2$).
	The following lemma states that the representations $\rho_1$ and $\rho_2$ are equivalent.
	\begin{lem}\label{iso-rep}
		The isomorphism of vector spaces 
		\begin{eqnarray*}
			f  :\ \ \ Cl_{0,2}&\longrightarrow& Cl_{1,2}^{0} \\  \ \ \ \ \ \  q_0+q_1I+J(q_2-Iq_3)&\longmapsto& q_0+q_1I+iJ(q_2-Iq_3)
		\end{eqnarray*}
		induces a $\mathbb{C}$-linear isomorphism between the representations $\rho_1$ and $\rho_2$, where the complex structure in both spaces is given by the right multiplication by $I$.
	\end{lem}
	\begin{proof}
		A computation shows that $f(gq)=gf(q)$ $\forall g\in Spin(0,2),\forall q\in Cl_{0,2}.$
	\end{proof}
	Identifying $e_0$ with $iI$ (see Remark \ref{rmk app cl12}), the isomorphism $f$ satisfies the following:
	\begin{lem}\label{iden}
		If $x\in \R^2$ and $q \in \mathbb{H}$ then $f(x\cdot q)=ie_0\cdot x\cdot f(q)$,
		where $i$ is the complex structure in $Cl_{1,2}^0$ given by the right multiplication by $I$.
	\end{lem}
	Since $E$ is trivial the orthonormal frame bundle is $Q_E=M\times \{1\}$ and $\widetilde{Q}_{E}=M\times\{\pm 1\}$ is also trivial. We consider the global section $\widetilde{s}_E:M\longrightarrow \{\pm 1\}$ given by $m\longmapsto+1$. Considering the inclusion $\widetilde{Q}_M\longrightarrow \widetilde{Q}_M\times \widetilde{Q}_E$, $\widetilde{s}_M\longmapsto(\widetilde{s}_M,\widetilde{s}_E)$ and the isomorphism $f$ of Lemma \ref{iden} we get the bundle isomorphism
	\begin{eqnarray*}\label{id_riemanniana*}		\Sigma M:=	\overset{\sim}{Q}_{_M}\times_\rho \Sigma_2&\longrightarrow &\Sigma_0:=\left({\overset{\sim}{Q}_{_M}\times_{_M}\overset{\sim}{Q}_{_E}}\right) \times_\rho Cl_{1,2}^{ ^\circ}
		\\
		\psi:=\big [ \overset{\sim}{s}_{_M},q\big ]&\longmapsto& \psi^*:=\big [ (\overset{\sim}{s}_{_M},\overset{\sim}{s}_{_E}),f(q)\big ];\nonumber
	\end{eqnarray*}
	it satisfies
	$$
	(\nabla_{_X}\psi)^*=\nabla_{_X}\psi^*,\hspace{.3cm}
	(X\cdot_M\psi)^*=iN\cdot X\cdot \psi^*\hspace{.3cm}\mbox{and}\hspace{.3cm}
	|\psi^+|^2-|\psi^-|^2=\langle\langle\psi^*,\psi^*\rangle\rangle
	$$
	for all $X\in TM,$ where $N=[(\overset{\sim}{s}_M,\overset{\sim}{s}_E), e_0]$ (see \cite[Prop. 3.4.9]{TBZJ} for more details).
	
	Let us see now the Lorentzian case. Here the spinor bundle of $M$ is $\Sigma M=\widetilde{Q}_M\times_{\rho_1} Cl_{_{1,1}},$ where $\rho_1:Spin(1,1)\longrightarrow Gl(Cl_{1,1})$ is the representation given by left-multiplication. As above we have $\widetilde{Q}_E=M\times \{\pm 1\},$ the section $\overset{\sim}{s}_{_E}: m\mapsto +1$ of $\widetilde{Q}_E,$ and the inclusion $\overset{\sim}{Q}_{_M}\longrightarrow \overset{\sim}{Q}_{_M}\times \overset{\sim}{Q}_{_E}$, $\overset{\sim}{s}_{_M}\longmapsto (\overset{\sim}{s}_{_M},\overset{\sim}{s}_{_E}).$ Using the isomorphism $Cl_{1,1}\simeq Cl_{1,2}^{0}$ (induced by the Clifford application $\R^{1,1}\longrightarrow Cl _{12}^{0}$, $x\longmapsto x\cdot e_2$, where $(e_0,e_1,e_2)$ is the standard basis of $\R^{1,2}$) we obtain a bundle isomorphism
	\begin{eqnarray*}\label{iso_estrella_lorentz}
		\overset{\sim}{Q}_{_M}\times_{\rho_1} Cl_{11}&\longrightarrow& \left({\overset{\sim}{Q}_{_M}\times_{_M}\overset{\sim}{Q}_{_E}}\right) \times_\rho Cl_{1,2}^{0}\\
		\psi&\longmapsto & \psi^*;\nonumber\end{eqnarray*}
	it satisfies the properties
	$$(\nabla_{_X}\psi)^*=\nabla_{_X}\psi^*,\hspace{.3cm}
	(X\cdot_M\psi)^*=X\cdot N \cdot \psi^*\hspace{.3cm}\mbox{and}\hspace{.3cm}
	|\psi^+|^2-|\psi^-|^2=\langle\langle\psi^*,\psi^*\rangle\rangle$$
	for all $X\in TM,$ where $N=[(\overset{\sim}{s}_M,\overset{\sim}{s}_E), e_2]$ (see \cite[Prop. 3.4.16]{TBZJ} for more details).
	
	\section{Bivectors and linear operators}\label{apendiceB}
	We prove that a skew-symmetric operator $u:\R^{r,s}\longrightarrow \R^{r,s}$ is identified with a bivector $\underline{u}\in \Lambda^2(\R^{r,s})$. Let us consider the following bracket in the Clifford algebra $Cl_{r,s}$
	\begin{equation}
	[a,b]=\frac{1}{2}(a\cdot b-b\cdot a),
	\end{equation}
	for all $a,b\in Cl_{r,s}$.
	
	\begin{lem}\label{lemaB1}
		Let $u:\R^{r,s}\longrightarrow \R^{r,s}$ be a skew-symmetric operator. The bivector that represents
		$u$ is
		\begin{eqnarray}
		\underline{u}=\frac{1}{2}\sum_{j=1}^{r+s}\varepsilon_je_j\cdot u(e_j),& & \varepsilon_j=\langle e_j,e_j\rangle=\pm 1,
		\end{eqnarray}
		and for all $\xi\in \R^{r,s}$ we have
		\begin{eqnarray}\label{representacion_u}
		[\underline{u},\xi]=u(\xi).
		\end{eqnarray}
	\end{lem}
	\begin{proof} Let us consider the linear application $u:\R^{r,s}\longrightarrow \R^{r,s}$ given by $e_i\longmapsto \varepsilon_j e_j$ and $e_j\longmapsto -\varepsilon_i e_i$   if $i<j$ and $e_k\longmapsto 0$ if $k\neq i,j$ which corresponds to $\varepsilon_i e_i\wedge \varepsilon_j e_j \in \Lambda^2 \R^{r,s}$. This map is skew-symmetric and $\underline{u}=\varepsilon_i\varepsilon_je_i\cdot e_j=\frac{1}{2}\varepsilon_i\varepsilon_j\left(e_i\cdot e_j-e_j\cdot e_i\right)$ satisfies
		\begin{eqnarray*}
			[\underline{u},e_k]&=&\frac{1}{2}\varepsilon_i\varepsilon_j(e_i\cdot e_j\cdot e_k-e_k\cdot e_i\cdot e_j),
		\end{eqnarray*}
		which yields for $k=i$
		$$[\underline{u},e_i]=\frac{1}{2}\varepsilon_i\varepsilon_j(\varepsilon_i e_j+\varepsilon_i e_j)=\varepsilon_j e_j=u(e_i);$$
		similarly we can prove that $[\underline{u},e_j]=u(e_j)$ and also readily see that $[\underline{u},e_k]=0$ for $k\neq i, j$. Equality \eqref{representacion_u} is a consequence of linearity.
	\end{proof}
	
	\textbf{Acknowledgments}. The author was partially supported by the project PAPIIT IA106218. She thanks P. Bayard for valuable suggestions during the development of this work and J. Roth for useful conversations, especially about the $\mathbb{L}(\kappa,\tau)$ spaces.

	\begin{center}
		E-mail address: bzavala@ciencias.unam.mx
		\\
		\small{FACULTAD DE CIENCIAS, UNIVERSIDAD NACIONAL AUT\'ONOMA DE M\'EXICO, AV. UNIVERSIDAD 3000, CIRCUITO EXTERIOR S/N, DELEGACI\'ON COYOAC\'AN, C.P. 04510, CIUDAD UNIVERSITARIA, CDMX, M\'EXICO.}
	\end{center}
	
\end{document}